\documentclass[pagesize,pdftex]{scrartcl}
\usepackage{latexsym} 
\usepackage[intlimits]{amsmath}
\usepackage{amsthm}
\usepackage{amsfonts}
\usepackage{amssymb}
\usepackage{amsxtra}
\usepackage{amscd}
\usepackage{ifthen}
\usepackage{graphicx}
\usepackage[shortlabels]{enumitem}
\usepackage{mathrsfs}
\usepackage[pagebackref=true]{hyperref}

\hypersetup{
pdfauthor={Giovanni P. Galdi and Mads Kyed},
pdftitle={Time-period flow of a viscous liquid past a body},
breaklinks=true,
colorlinks=true,
linkcolor=blue,
citecolor=blue,
urlcolor=blue,
filecolor=blue,
}

\numberwithin{equation}{section} 
\setkomafont{title}{\normalfont}

\newenvironment{pdeq}{ \left\{ \begin{aligned}}{\end{aligned}\right.}

\newcommand{\eqrefsub}[2]{\ensuremath{\eqref{#1}_{#2}}}
\newcommand{\np}[1]{(#1)}
\newcommand{\nb}[1]{[#1]}
\newcommand{\bp}[1]{\big(#1\big)}
\newcommand{\bb}[1]{\big[#1\big]}
\newcommand{\Bp}[1]{\bigg(#1\bigg)}
\newcommand{\Bpb}[1]{\bigg(#1\bigg]}

\newcommand{\Bb}[1]{\bigg[#1\bigg]}
\newcommand{\calb}{{\mathcal B}}

\newcommand{\calf}{{\mathcal F}}

\newcommand{\call}{{\mathcal L}}

\newcommand{\calp}{{\mathcal P}}

\newcommand{\calt}{{\mathcal T}}

\newcommand{\R}{\mathbb{R}}

\newcommand{\Z}{\mathbb{Z}}

\newcommand{\N}{\mathbb{N}}
\DeclareMathOperator{\e}{e}
\newcommand{\B}{B}

\DeclareMathOperator{\id}{Id}
\DeclareMathOperator{\Div}{div}

\DeclareMathOperator{\supp}{supp}

\DeclareMathOperator{\trace}{Tr}

\newcommand{\embeds}{\hookrightarrow}

\newcommand{\Ra}{\Rightarrow}

\newcommand{\ra}{\rightarrow}

\newcommand{\set}[1]{\ensuremath{\{#1\}}}

\newcommand{\setc}[2]{\ensuremath{\{#1\ \lvert\ #2\}}}
\newcommand{\setcl}[2]{\ensuremath{\bigl\{#1\ \lvert\ #2\bigr\}}}
\newcommand{\setcL}[2]{\ensuremath{\biggl\{#1\ \lvert\ #2\biggr\}}}

\newcommand{\closure}[2]{\overline{#1}^{#2}}
\newcommand{\seqkN}[1]{\ensuremath{\set{#1_k}_{k=1}^\infty}}
\newcommand{\proj}{\calp}
\newcommand{\projcompl}{\calp_\bot}
\newcommand{\hproj}{\calp_H}
\newcommand{\quotientmap}{\pi}
\newcommand{\grp}{G}
\newcommand{\dualgrp}{\widehat{G}}

\newcommand{\torus}{{\mathbb T}}

\newcommand{\RthreeR}{\R\times\R^3}

\newcommand{\OmegaR}{\R\times\Omega}
\newcommand{\partialOmegaR}{\R\times\partial\Omega}
\newcommand{\boundeddomain}{D}
\newcommand{\boundeddomainR}{\R\times\boundeddomain}
\newcommand{\partialboundeddomainR}{\R\times\partial\boundeddomain}
\newcommand{\grad}{\nabla}

\newcommand{\pdn}[1]{\frac{\partial #1}{\partial n}}

\newcommand{\dx}{{\mathrm{d}}x}

\newcommand{\ds}{{\mathrm d}s}
\newcommand{\dt}{{\mathrm d}t}

\newcommand{\dS}{{\mathrm d}S}

\newcommand{\TDRper}{\mathscr{S}^\prime_{per}}

\newcommand{\linf}[2]{\langle #1, #2\rangle}

\newcommand{\FT}{\mathscr{F}}
\newcommand{\iFT}{\mathscr{F}^{-1}}

\newcommand{\norm}[1]{\lVert#1\rVert}
\newcommand{\homnorm}[1]{\langle#1\rangle}

\newcommand{\snorm}[1]{{\lvert #1 \rvert}}

\newcommand{\snormL}[1]{{\Bigl\lvert #1 \Big\rvert}}

\newcommand{\WSR}[2]{W^{#1,#2}} 
 
\newcommand{\DSR}[2]{D^{#1,#2}} 
 
\newcommand{\DSRRzero}[2]{D^{#1,#2}_{\Rzero}} 

\newcommand{\CR}[1]{C^{#1}}  
\newcommand{\LR}[1]{L^{#1}}

\newcommand{\LRloc}[1]{L^{#1}_{loc}} 
\newcommand{\CRi}{\CR \infty}
\newcommand{\CRci}{\CR \infty_0}

\newcommand{\LRsigma}[1]{L^{#1}_{\sigma}}

\newcommand{\LRper}[1]{L^{#1}_{\mathrm{per}}}
\newcommand{\LRcomplper}[1]{L^{#1}_{\mathrm{per},\bot}}
\newcommand{\WSRper}[2]{W^{#1,#2}_{\mathrm{per}}} 
 
\newcommand{\DSRper}[2]{D^{#1,#2}_{\mathrm{per}}} 
\newcommand{\DSRcomplper}[2]{D^{#1,#2}_{\mathrm{per},\bot}} 
\newcommand{\DSRcomplperRzero}[2]{D^{#1,#2}_{\mathrm{per},\bot,\Rzero}} 
\newcommand{\WSRcomplper}[2]{W^{#1,#2}_{\mathrm{per},\bot}} 
\newcommand{\CRper}{\CR{}_{\mathrm{per}}}
\newcommand{\CRiper}{\CR{\infty}_{\mathrm{per}}}
\newcommand{\CRciper}{\CR{\infty}_{0,\mathrm{per}}}
\newcommand{\CRcicomplper}{\CR{\infty}_{0,\mathrm{per},\bot}}

\newcommand{\WSRpercomplsigmazero}[2]{\mathcal{W}^{#1,#2}_{\mathrm{per},\bot}}

\newcommand{\xoseen}[1]{{X}^{#1}_\rey}
\newcommand{\xoseensigmazero}[1]{{\mathcal{X}}^{#1}_{\rey}}

\newcommand{\bspace}{X}

\newcommand{\nsnonlinb}[2]{#1\cdot\grad #2}
\newcommand{\nsnonlin}[1]{\nsnonlinb{#1}{#1}}
\newcommand{\vvel}{v}
\newcommand{\vpres}{p}

\newcommand{\Vvel}{V}
\newcommand{\Vpres}{P}

\newcommand{\adjvel}{\psi}
\newcommand{\adjpres}{\eta}
\newcommand{\wvel}{w}
\newcommand{\wpres}{\pi}
\newcommand{\Wvel}{W}
\newcommand{\Wpres}{\Pi}
\newcommand{\twvel}{\tilde{w}}
\newcommand{\twpres}{\tilde{\pi}}

\newcommand{\uvel}{u}

\newcommand{\uvelbrd}{u_*}
\newcommand{\upres}{\mathfrak{p}}
\newcommand{\uvelinfty}{u_\infty}

\newcommand{\Uvel}{U}

\newcommand{\Upres}{\mathfrak{P}}
\newcommand{\tuvel}{\tilde{u}}
\newcommand{\tupres}{\tilde{\mathfrak{p}}}

\newcommand{\ALOseen}{A_{\mathrm{O}}}
\newcommand{\ALOseeninverse}{A_{\mathrm O}^{-1}}

\newcommand{\liftWvel}{{\mathcal{W}}}

\newcommand{\liftWpres}{\Pi_\bot}
\newcommand{\liftVvel}{{\mathcal{V}}}
\newcommand{\liftVpres}{\Pi_s}
\newcommand{\fundsollaplace}{\varGamma_{\text{\tiny{L}}}}
\newcommand{\tin}{\text{in }}
\newcommand{\tif}{\text{if }}
\newcommand{\ton}{\text{on }}

\newcommand{\tand}{\text{and }}

\newcommand{\half}{\frac{1}{2}}
\newcommand{\fourth}{\frac{1}{4}}
\renewcommand{\epsilon}{\varepsilon}
\renewcommand{\phi}{\varphi}
\newcommand{\rey}{\lambda}

\newcommand{\tay}{\calt}
\newcommand{\per}{\tay}
\newcommand{\iper}{\frac{1}{\tay}}
\newcommand{\perf}{\frac{2\pi}{\tay}}

\newcommand{\eone}{\e_1}

\newcommand{\bigo}{O}

\newcommand{\linoproseen}{\call}

\newcommand{\cutoff}{\chi}

\newcommand{\Rzero}{{R_0}}
\newcommand{\Rstar}{{R_*}}
\newcommand{\Rmiddel}{{R}}
\newcommand{\rhoover}{{\rho}}
\newcommand{\cutoffover}{{\psi_2}}
\newcommand{\rhounder}{{\rho}}
\newcommand{\cutoffunder}{{\psi_1}}

\newcommand{\uscalar}{u}
\newcommand{\bodyvel}{v_b}
\newcommand{\F}{F}
\newcommand{\f}{f}
\newcommand{\testphi}{\phi}
\newcommand{\uvelinftyscalar}{\uvelinfty}
\newcommand{\rhsvvel}{\mathfrak{R}_1}
\newcommand{\rhswvel}{\mathfrak{R}_2}
\newcommand{\fixedpointmap}{\mathcal{N}}
\newcommand{\fixedpointspace}[1]{\mathcal{K}_{\rey}^{#1}}
\newcommand{\newCCtr}[2][d]{
\newcounter{#2}\setcounter{#2}{0}
\expandafter\xdef\csname kyedtheconst#2\endcsname{#1}
}
\newcommand{\Cc}[2][nolabel]{
\stepcounter{#2}
\expandafter\ensuremath{\csname kyedtheconst#2\endcsname_{\arabic{#2}}}
\ifthenelse{\equal{#1}{nolabel}}
{}
{\expandafter\xdef\csname kyedconst#1\endcsname
{\expandafter\ensuremath{\csname kyedtheconst#2\endcsname_{\arabic{#2}}}}}
}
\newcommand{\Ccn}[2][nolabel]{
\expandafter\ensuremath{\csname kyedtheconst#2\endcsname}
\ifthenelse{\equal{#1}{nolabel}}
{}
{\expandafter\xdef\csname kyedconst#1\endcsname
{\expandafter\ensuremath{\csname kyedtheconst#2\endcsname}}}
}
\newcommand{\CcSetCtr}[2]{
\setcounter{#1}{#2}
}
\newcommand{\Cclast}[1]{
\expandafter\ensuremath{\csname kyedtheconst#1\endcsname_{\arabic{#1}}}
}
\newcommand{\Ccllast}[1]{
\addtocounter{#1}{-1}
\expandafter\ensuremath{\csname kyedtheconst#1\endcsname_{\arabic{#1}}}
\addtocounter{#1}{1}
}
\newcommand{\const}[1]{
\expandafter{\ifcsname kyedconst#1\endcsname
  \csname kyedconst#1\endcsname
\else
  \errmessage{Undefined Kyedconstant #1.}%
\fi}
}
\renewcommand{\eqrefsub}[2]{\eqref{#1}\textsubscript{#2}}

\theoremstyle{plain}
\newtheorem{thm}{Theorem}[section]

\newtheorem{lem}[thm]{Lemma}

\newtheorem{cor}[thm]{Corollary}
\theoremstyle{remark}
\newtheorem{rem}[thm]{Remark}

\begin{document}
\title{Time-period flow of a viscous liquid past a body}

\author{
Giovanni P. Galdi\thanks{Partially supported by NSF-DMS grant 1614011.}\\ 
Department of Mechanical Engineering and Materials Science\\
University of Pittsburgh\\
Pittsburgh, PA 15261, USA\\
Email: \texttt{galdi@pitt.edu}
\and
Mads Kyed\\ 
Fachbereich Mathematik\\
Technische Universit\"at Darmstadt\\
Schlossgartenstr. 7, 64289 Darmstadt, Germany\\
Email: \texttt{kyed@mathematik.tu-darmstadt.de}\\
}

\date{\today}
\maketitle

\begin{abstract}
Time-periodic solutions to 
the Navier-Stokes equations that govern the flow of a viscous liquid past a three-dimensional body moving with a time-periodic velocity 
are investigated.
The net motion of the body over a full time-period is assumed to be non-zero.
In this case the appropriate linearization of the system is constituted by the time-periodic Oseen equations in a three-dimensional exterior domain.
A priori $\LR{q}$ estimates are established for this linearization. Based on these estimates, 
existence of a solution to the fully non-linear Navier-Stokes problem is obtained by the contraction mapping principle.  
\end{abstract}


\newCCtr[C]{C}
\newCCtr[c]{c}
\let\oldproof\proof
\def\proof{\CcSetCtr{c}{-1}\oldproof} 
\newCCtr[M]{M}
\newCCtr[B]{B}
\newCCtr[\epsilon]{eps}
\CcSetCtr{eps}{-1}

\section{Introduction}
Consider a three-dimensional body moving with a prescribed time-periodic velocity $\bodyvel(t)$ in a viscous liquid
governed by the Navier-Stokes equations.  
If the body occupies a bounded, simply connected domain $\calb\subset\R^3$, the liquid flowing past it occupies the corresponding exterior domain 
$\Omega:=\R^3\setminus\calb$.
The motion of the liquid can then be described, in a coordinate system attached to the body, by the following system of equations: 
\begin{align}\label{intro_nspastbody}
\begin{pdeq}
&\partial_t\uvel + \nsnonlin{\uvel} = \nu\Delta\uvel - \grad\upres + f && \tin\OmegaR,\\
&\Div\uvel =0 && \tin\OmegaR,\\
&\uvel=\uvelbrd&& \ton\partialOmegaR,\\
&\lim_{\snorm{x}\ra\infty} \uvel(t,x) = -\bodyvel(t).
\end{pdeq}
\end{align} 
Here, $\uvel:\OmegaR\ra\R^3$ denotes the Eulerian velocity field and $\upres:\OmegaR\ra\R$ the pressure field
of the liquid. 
As is natural for time-periodic problems, the time axis is taken to be the whole of $\R$, and 
so $(t,x)\in\OmegaR$ denotes the time variable $t$ and spatial variable $x$ of the system, respectively.
A body force $f:\OmegaR\ra\R^3$ and velocity distribution 
$\uvelbrd:\partialOmegaR\ra\R^3$ of the liquid on the surface of the body have been included. 
The constant coefficient of kinematic viscosity of the liquid is denoted by $\nu$.
An investigation of time-periodic solutions, that is, solutions $(\uvel,\upres)$ satisfying for some fixed $\per>0$
\begin{align}\label{intro_tpcondsol}
\uvel(t+\per,x) = \uvel(t,x),\quad
\upres(t+\per,x) = \upres(t,x),
\end{align}
corresponding to time-periodic data of the same period,
\begin{align}\label{intro_tpconddat}
\bodyvel(t+\per)=\bodyvel(t),\quad f(t+\per,x) = f(t,x),\quad\uvelbrd(t+\per,x) = \uvelbrd(t,x),
\end{align}
will be carried out. 

We assume the net motion of the body over a full time-period is non-zero, that is, 
\begin{align}\label{intro_uvelinftyneqzero}
\int_0^\per \bodyvel(t)\,\dt\neq 0.
\end{align}
We shall not treat the case of a vanishing net motion.
The distinction between the two cases is justified by the physics of the problem. 
In the former case, the body performs a nonzero translatory motion, which induces a wake in the region behind it.  
In the latter case, the motion of the body would be purely oscillatory without a wake. The different properties 
of the solutions in the two cases also influence the mathematical analysis of the problem. If the net motion of the body has a 
nonzero translatory component, the appropriate linearization
of \eqref{intro_nspastbody} is a time-periodic Oseen system. If the net motion over a period is zero, the linearization is a time-periodic Stokes system.  
The investigation in this paper is based on suitable $\LR{q}$ estimates for solutions to the time-periodic Oseen system. Similar estimate do not hold for the corresponding Stokes 
system, in which case a different approach is needed.
It will further be assumed that the motion of the body is directed along a single axis, say
\begin{align}\label{intro_bodyvelsingleaxis}
\bodyvel(t) = \uvelinfty(t)\eone,\quad \uvelinfty(t)\in\R.
\end{align}
This assumption is made for technical reasons only.
 
We shall in Theorem \ref{ext_ExistenceNonlinProbThm} establish existence of a solution to \eqref{intro_nspastbody} for data $f$, $\bodyvel$ and $\uvelbrd$ sufficiently restricted in ``size''. The solution is strong both in the sense of local regularity and global summability.
The proof is based on the contraction mapping principle and suitable $\LR{q}$ estimates 
of solutions to a linearization of \eqref{intro_nspastbody}. More specifically, we linearize \eqref{intro_nspastbody} around $\bodyvel$ and 
obtain, due to \eqref{intro_uvelinftyneqzero}, a time-periodic Oseen system. In Theorem \ref{lin_linsystem} and Corollary \ref{ext_lqestThmFull}
we identify a time-periodic Sobolev-type space that is mapped homeomorphically onto a time-periodic $\LR{p}$ space by the Oseen operator.
We then employ embedding properties of the Sobolev-type space to show existence of a solution to the fully nonlinear problem \eqref{intro_nspastbody}
by a fixed-point argument. A similar result was obtained for a two-dimensional exterior domain in \cite{GaldiTPFlowMovingCylinder}.

The study of time-periodic solutions to the Navier-Stokes equations was originally suggested by \textsc{Serrin} \cite{Serrin_PeriodicSolutionsNS1959}. 
The first rigorous investigations of the classical time-periodic Navier-Stokes problem in bounded domains are due to \textsc{Prodi} \cite{Prodi1960}, \textsc{Yudovich} \cite{Yudovich60} and \textsc{Prouse} \cite{Prouse63}. 
Further properties and extensions to other types of domains and problems have been studied by a number of authors over the years:
\cite{KanielShinbrot67},
\cite{Takeshita69},
\cite{Morimoto1972}, 
\cite{MiyakawaTeramoto82},
\cite{Teramoto1983},
\cite{Maremonti_TimePer91},
\cite{Maremonti_HalfSpace91},
\cite{MaremontiPadula96},
\cite{KozonoNakao96},
\cite{Yamazaki2000},
\cite{GaldiSohr2004},
\cite{GaldiSilvestre_Per06},
\cite{GaldiSilvestre_Per09},
\cite{Taniuchi2009},
\cite{BaalenWittwer2011},
\cite{Silvestre_TPFiniteKineticEnergy12},
\cite{GaldiTP2D12},
\cite{KyedExRegTPNS2014}, \cite{KyedMaxregTPNS2014}, \cite{Kyed_FundsolTPStokes2016}
\cite{Nguyen2014},
\cite{HieberTP2015}.
Of these articles, \cite{GaldiSilvestre_Per06, GaldiSilvestre_Per09, KyedMaxregTPNS2014, KyedExRegTPNS2014, GaldiTP2D12, Kyed_FundsolTPStokes2016}
treat the same type of flow past a body \eqref{intro_nspastbody} that is investigated in the following. While weak solutions to an even more general problem are established in \cite{GaldiSilvestre_Per06, GaldiSilvestre_Per09}, the corresponding whole-space problem in dimension two and three is studied in \cite{GaldiTP2D12,KyedExRegTPNS2014,KyedMaxregTPNS2014, Kyed_FundsolTPStokes2016}.

\section{Notation}

Constants in capital letters in the proofs and theorems are global, while constants in small letters are local to the proof in which they appear.
The notation $\Ccn{C}(\xi)$ is used to emphasize the dependence of a constant on a parameter $\xi$.

The notation $\B_R$ is used to denote balls in $\R^n$ centered at $0$ with radius $R>0$.

The symbol $\Omega$ denotes an exterior domain of $\R^n$, that is, an open connected set that is the complement of the closure of a simply connected bounded domain $\calb\subset\R^n$. Without loss of generality, it is assumed that $0\in\calb$. 
Two constants $\Rzero>\Rstar>0$ with
$\calb\subset\subset\B_\Rstar$ remain fixed.
Moreover, the domains $\Omega_R:=\Omega\cap\B_R$, $\Omega_{R_1,R_2}:=\Omega\cap\B_{R_2}\setminus\B_{R_1}$ and
$\Omega^R:=\Omega\setminus\B_R$ are introduced.   

Points in $\R\times\R^n$ are denoted by $(t,x)$. Throughout, $t$ is referred to as the time  and $x$ as the spatial variable.
For a sufficiently regular function $u:\R\times\R^n\ra\R$, $\partial_i u:=\partial_{x_i} u$ denotes spatial derivatives.

\section{Preliminaries}

A framework shall be employed based on the function space of smooth and compactly supported $\per$-time-periodic 
functions:  
\begin{align*}
\CRciper(\R\times\overline{\Omega}) := \setcl{f\in\CRi(\R\times\Omega)}{f(t+\per,x)=f(t,x)\ \wedge\ f\in\CRci\bp{[0,\per]\times\overline{\Omega}}}.
\end{align*}

For simplicity, the interval $(0,\per)$ of one time period is sometimes denoted by $\torus$. 
An $\LR{q}$ norm on the time-space domain $\torus\times\Omega$ is defined by
\begin{align*}
\norm{f}_q:=\norm{f}_{q,\torus\times\Omega}:= \Bp{\iper\int_0^\per\int_{\Omega}\snorm{f(t,x)}^q\,\dx\dt}^{\frac{1}{q}},\quad q\in[1,\infty).
\end{align*}
Lebesgue spaces of time-periodic functions are defined by  
\begin{align*}
&\LRper{q}(\OmegaR):= \closure{\CRciper(\R\times\overline{\Omega})}{\norm{\cdot}_{q}}.
\end{align*}
It is easy to see that the elements of $\LRper{q}(\OmegaR)$ coincide with the $\per$-time-periodic extension of functions in $\LR{q}((0,\per)\times\Omega)$.
The Lebesgue space $\LR{q}(\Omega)$ is treated as the subspace of functions in
$\LRper{q}(\OmegaR)$ that are time-independent. 
For such functions the  $\LR{q}(\Omega)$ norm coincides with the norm $\norm{f}_q$ introduced above. 

Sobolev spaces of $\per$-time-periodic functions are also introduced
as completions of $\CRciper(\R\times\overline{\Omega})$
in appropriate norms:
\begin{align*}
&\WSRper{1,2}{q}(\OmegaR) := \closure{\CRciper(\R\times\overline{\Omega})}{\norm{\cdot}_{1,2,q}},\quad 
\norm{\uvel}_{1,2,q}:=\Bp{\sum_{\snorm{\beta}\leq 1} \norm{\partial_t^\beta \uvel}_{q}^q+\sum_{\snorm{\alpha}\leq 2} \norm{\partial_x^\alpha \uvel}_{q}^q}^{\frac{1}{q}}.
\end{align*}
These Sobolev spaces are clearly subspaces of the classical anisotropic Sobolev spaces 
$\WSR{1,2}{q}((0,\per)\times\Omega)$.
Analogously, homogeneous Sobolev spaces of time-periodic functions
\begin{align*}
&\DSRper{1}{q}(\OmegaR) := \closure{\CRciper(\R\times\overline{\Omega})}{\homnorm{\cdot}_{1,q}},\quad 
\homnorm{\upres}_{1,q}:=\norm{\grad \upres}_{q} + {\iper\int_0^\per\snormL{\int_{\Omega_\Rzero}\upres(t,x)\,\dx}\dt}
\end{align*}
are defined. It is easy to see that $\DSRper{1}{q}(\OmegaR)$ can be identified with $\LR{q}\bp{(0,\per);\DSR{1}{q}(\Omega)}$, where
$\DSR{1}{q}(\Omega)$ is the classical homogeneous Sobolev space. 

In a similar manner, Lebesgue and Sobolev spaces of time-periodic vector-valued functions are defined for any Banach space $\bspace$
respectively as
\begin{align*}
\LRper{q}\bp{\R;\bspace}:=\closure{\CRiper\bp{\R;\bspace}}{\norm{\cdot}_{\LR{q}((0,\per);\bspace)}},\quad 
\WSRper{m}{q}\bp{\R;\bspace}:=
\closure{\CRiper\bp{\R;\bspace}}{\norm{\cdot}_{\WSR{m}{q}((0,\per);\bspace)}}.
\end{align*} 
One readily verifies that $\LRper{q}\bp{\R;\bspace}$ coincides with the $\per$-periodic extensions of functions in the Lebesgue space 
$\LRper{q}\bp{(0,\per);\bspace}$.
One may further verify for sufficiently regular domains, say $\Omega$ of class $\CR{1}$, that
\begin{align*}
\WSRper{1,2}{q}(\OmegaR) &= \WSRper{1}{q}\bp{\R;\LR{q}(\Omega)}\cap\LRper{q}\bp{\R;\WSR{2}{q}(\Omega)} \\
&=\setc{\uvel\in\LRper{q}(\OmegaR)}{\norm{\uvel}_{1,2,q}<\infty}.
\end{align*}

Sufficiently regular $\per$-time-periodic functions $u:\R\times\R^n\ra\R$ can be decomposed into 
what will be referred to as a \emph{steady-state} part $\proj u:\R\times\R^n\ra\R$
and \emph{oscillatory} part $\projcompl u:\R\times\R^n\ra\R$ by
\begin{align}\label{intro_defofprojGernericExpression}
\proj u(t,x):=\iper\int_0^\per u(x,s)\,\ds\quad\tand\quad\projcompl u(t,x) := u(t,x)-\proj u(t,x)
\end{align}
whenever these expressions are well-defined. 
Note that the steady-state part of a $\per$-time-periodic function $u$ is time-independent,
and the oscillatory part $\projcompl u$ has vanishing time-average over the period.
Also note that $\proj$ and $\projcompl$ are complementary projections, that is, $\proj^2=\proj$ and $\projcompl=\id-\proj$.
Based on these projections, the following sub-spaces are defined: 
\begin{align*}
&\CRcicomplper(\OmegaR):=\setc{f\in\CRciper(\OmegaR)}{\proj f = 0},\\
&\LRcomplper{q}(\OmegaR):=\setc{f\in\LRper{q}(\OmegaR)}{\proj f = 0},\\
&\WSRcomplper{1,2}{q}(\OmegaR):= \setc{\uvel\in\WSRper{1,2}{q}(\OmegaR)}{\proj \uvel = 0}.
\end{align*}
Intersections of these spaces are denoted by
\begin{align*}
&\LRcomplper{q,r}(\OmegaR):=\LRcomplper{q}(\OmegaR)\cap\LRcomplper{r}(\OmegaR),\\
&\WSRcomplper{1,2}{q,r}(\OmegaR):=\WSRcomplper{1,2}{q}(\OmegaR)\cap\WSRcomplper{1,2}{r}(\OmegaR)
\end{align*}
and equipped with the canonical norms. 
Similar subspaces of homogeneous Sobolev spaces are defined by
\begin{align*}
&\DSRcomplper{1}{q}(\OmegaR):= \setcL{\upres\in\DSRper{1}{q}(\OmegaR)}{\proj \upres = 0},\\
&\DSRcomplper{1}{q,r}(\OmegaR):= \DSRcomplper{1}{q}(\OmegaR)\cap\DSRcomplper{1}{r}(\OmegaR),\\
&\DSRcomplperRzero{1}{q}(\OmegaR):= \setcL{\upres\in\DSRper{1}{q}(\OmegaR)}{\proj \uvel = 0\ \wedge\ \int_{\Omega_\Rzero}\upres(t,x)\,\dx=0},\\
&\DSRcomplperRzero{1}{q,r}(\OmegaR) := \DSRcomplperRzero{1}{q}(\OmegaR)\cap \DSRcomplperRzero{1}{r}(\OmegaR).
\end{align*}
All spaces above are clearly Banach spaces.

Finally, we introduce for $\rey>0$ and $q\in(1,2)$
the Sobolev-type space 
\begin{align*}
&\xoseen{q}(\Omega):=\setc{\vvel\in\LRloc{q}(\Omega)}{\norm{\vvel}_{\xoseen{q}}<\infty},\\
&\norm{\vvel}_{\xoseen{q}} := 
\rey^{\half}\norm{\vvel}_{\frac{2q}{2-q}} + \rey^{\fourth}\norm{\grad\vvel}_{\frac{4q}{4-q}} +
\rey\norm{\partial_1\vvel}_q + \norm{\grad^2\vvel}_q, 
\end{align*}
which is used to characterize the velocity field of a steady-state Oseen system in three-dimensional exterior domains.
An appropriate function space for the corresponding pressure term is the homogeneous Sobolev space
\begin{align*}
\DSRRzero{1}{q}(\Omega):=\setcL{\vvel\in\DSR{1}{q}(\Omega)}{\int_{\B_\Rzero}\vpres\,\dx=0}
\end{align*}
equipped with the norm $\homnorm{\cdot}_{1,q}:=\norm{\grad\cdot}_q$. As mentioned above, $\DSR{1}{q}(\Omega)$ 
denotes the classical homogeneous Sobolev space.

\section{An Embedding Theorem}

Embedding properties of the Sobolev spaces of time-periodic functions defined in the previous section shall 
be established. As such properties may be of use in other applications as well, we consider in this section an exterior domain $\Omega\subset\R^n$ of arbitrary dimension $n\geq 2$.

\begin{thm}\label{SobEmbeddingThm}
Let $\Omega\subset\R^n$ ($n\geq 2)$ be an exterior domain of class $\CR{1}$ and $q\in(1,\infty)$. 
Assume that $\alpha\in\bb{0,2}$ and $p_0,r_0\in[q,\infty]$ satisfy
\begin{align}\label{SobEmbeddingThm_Condp0}
\begin{pdeq}
&r_0\leq \frac{2q}{2-\alpha q} && \tif\ \alpha q<2,\\
&r_0<\infty && \tif\ \alpha q =2,\\
&r_0\leq\infty && \tif\ \alpha q >2,
\end{pdeq}
\qquad
\begin{pdeq}
&p_0\leq \frac{nq}{n-(2-\alpha) q} && \tif\  \np{2-\alpha}q<{n},\\
&p_0<\infty && \tif\ \np{2-\alpha}q={n},\\
&p_0\leq\infty && \tif\ \np{2-\alpha}q>{n},
\end{pdeq}
\end{align} 
and that $\beta\in\bb{0,1}$ and $p_1,r_1\in[q,\infty]$ satisfy
\begin{align}\label{SobEmbeddingThm_Condp1}
\begin{pdeq}
&r_1\leq \frac{2q}{2-\beta q} && \tif\ \beta q<2,\\
&r_1<\infty && \tif\ \beta q =2,\\
&r_1\leq\infty && \tif\ \beta q >2,
\end{pdeq}
\qquad
\begin{pdeq}
&p_1\leq \frac{nq}{n-(1-\beta) q} && \tif\  \np{1-\beta}q<{n},\\
&p_1<\infty && \tif\ \np{1-\beta}q={n},\\
&p_1\leq\infty && \tif\ \np{1-\beta}q>{n}.
\end{pdeq}
\end{align} 
Then
\begin{align}\label{SobEmbeddingThm_Est}
\forall\uscalar\in\WSRper{1,2}{q}(\OmegaR):\quad \norm{\uscalar}_{\LRper{r_0}\np{\R;\LR{p_0}(\Omega)}} +\norm{\grad\uscalar}_{\LRper{r_1}\np{\R;\LR{p_1}(\Omega)}}  \leq \Cc{C}\norm{\uscalar}_{1,2,q},
\end{align}
with $\Cclast{C}=\Cclast{C}(\per,n,\Omega,r_0,p_0,r_1,p_1)$.
\end{thm}

\begin{proof}
The regularity of $\Omega$ is sufficient to ensure existence of a continuous extension operator 
$E:\WSRper{1,2}{q}(\OmegaR)\ra\WSRper{1,2}{q}(\R\times\R^n)$ as in the case of classical Sobolev spaces. 
Consequently, it suffices to show \eqref{SobEmbeddingThm_Est} 
for functions $\uscalar\in\WSRper{1,2}{q}(\R\times\R^n)$. For this purpose, we identify $\WSRper{1,2}{q}(\R\times\R^n)$ with 
a Sobolev space of functions defined on the group $\grp:=\bp{\R/{2\pi\Z}}\times\R^n$. Endowed with the quotient topology 
induced by the quotient map 
$\quotientmap:\R\times\R^n\ra\bp{\R/{2\pi\Z}}\times\R^n$, $\grp$ becomes a locally compact abelian group. A Haar measure is given by the product of the Lebesgue measure on $\R^n$ and the normalized Lebesgue measure on the interval $[0,\per)\simeq\R/\per\Z$. Also via the quotient map, the 
space of smooth functions on $\grp$ is defined as
\begin{align}\label{lt_smoothfunctionsongrp}
\CRi(\grp):=\setc{f:\grp\ra\R}{f\circ\quotientmap \in\CRi(\R\times\R^n)}.
\end{align}
Sobolev spaces $\WSR{1,2}{q}(\grp)$ can then be defined as 
the closure of $\CRci(\grp)$, the subspace of compactly supported functions in $\CRi(\grp)$, in the norms $\norm{\cdot}_{1,2,q}$. It is easy to verify that $\WSR{1,2}{q}(\grp)$ and $\WSRper{1,2}{q}(\R\times\R^n)$
are homeomorphic. Further details can be found in \cite{KyedMaxregTPNS2014} and \cite{habil}. 
The availability of the Fourier transform $\FT_\grp$ is an immediate advantage in the group setting. Denoting points in the dual group $\dualgrp:=\Z\times\R^n$ by $(k,\xi)$, we obtain the representation
\begin{align*}
\partial_j\projcompl\uvel 
&= \iFT_\grp\Bb{\frac{i\xi_j \bp{1-\delta_\Z(k)}}{\snorm{\xi}^2+i\perf k} \FT_\grp\bb{\partial_t\uvel-\Delta\uvel}},
\end{align*}
where $\delta_\Z$ denotes the delta distribution on $\Z$, that is, $\delta_{\Z}(0)=1$ and $\delta_{\Z}(k)=0$ for $k\neq 0$.
Since $\FT_\grp=\FT_{\R/{2\pi\Z}}\circ\FT_{\R^n}$, it follows that
\begin{align}\label{SobEmbeddingThm_ConvRep}
\partial_j\projcompl\uvel &= \iFT_{\R/{2\pi\Z}}\Bb{\np{1-\delta_\Z}\snorm{k}^{-\half\beta}}*_{\R/{2\pi\Z}}\iFT_{\R^n}\Bb{\snorm{\xi}^{\beta-1}}*_{\R^n}F,
\end{align}
with
\begin{align*}
F:=\iFT_\grp\Bb{ M(k,\xi)\, \FT_\grp\bb{\partial_t\uvel-\Delta\uvel}},\quad M(k,\xi):=\frac{\snorm{k}^{\half\beta}\,\snorm{\xi}^{1-\beta}\,i\xi_j \bp{1-\delta_\Z(k)}}{\snorm{\xi}^2+i\perf k}.
\end{align*}
Owing to the fact that $M$ has no singularities,
one can utilize a so-called transference principle and verify that $M$ is an $\LR{q}(\grp)$ Fourier multiplier for all $q\in(1,\infty)$; see
\cite{KyedMaxregTPNS2014} for the details on such an approach. 
It follows
that $F\in\LR{q}(\grp)$ with $\norm{F}_q \leq \Ccn{C}\,\norm{u}_{1,2,q}$. It is standard to compute the inverse Fourier transform $\gamma_\beta:=\iFT_{\R/{2\pi\Z}}\Bb{\np{1-\delta_\Z}\snorm{k}^{-\half\beta}}$ appearing on the right-hand side in \eqref{SobEmbeddingThm_ConvRep}. 
Choosing for example $[-\half\per,\half\per)$ as a realization of ${\R/{2\pi\Z}}$, one has $\gamma_\beta(t)=\Ccn{C}t^{-1+\half\beta} + h(t)$ with $h\in\CRi({\R/{2\pi\Z}})$; see for example \cite[Example 3.1.19]{Grafakos1}.
Clearly, $\gamma_\beta\in\LR{\frac{1}{1-\half\beta},\infty}({\R/{2\pi\Z}})$. Thus, by Young's inequality, see for example \cite[Theorem 1.4.24]{Grafakos1}, the mapping $\phi\ra \gamma_\beta *_{\R/{2\pi\Z}} \phi$ extends to a bounded operator from $\LR{q}({\R/{2\pi\Z}})$ into $\LR{r}({\R/{2\pi\Z}})$ for $r\in(1,\infty)$ satisfying
\begin{align}\label{SobEmbeddingThm_ExponentsCond1}
\frac{1}{r} = \Bp{1-\half\beta} + \frac{1}{q}-1.
\end{align}
The mapping $\phi\ra\iFT_{\R^n}\bb{\snorm{\xi}^{\beta-1}}*_{\R^n}\phi$ can be identified with the operator $\Delta_{\R^n}^{\frac{\beta-1}{2}}$. It is well-known, see for example \cite[Theorem 6.1.13]{Grafakos2}, that this operator is bounded from
$\LR{q}(\R^n)$ into $\LR{p}(\R^n)$ for $p\in(1,\infty)$ satisfying
\begin{align}\label{SobEmbeddingThm_ExponentsCond2}
\frac{1}{p} = \frac{1}{q}- \frac{1-\beta}{n}.
\end{align}
We now consider $r_1,p_1\in(1,\infty)$ that satisfy \eqref{SobEmbeddingThm_ExponentsCond1} and \eqref{SobEmbeddingThm_ExponentsCond2}. We then recall \eqref{SobEmbeddingThm_ConvRep} to estimate
\begin{align*}
\norm{\partial_j\projcompl\uvel}_{\LR{r_1}\np{{\R/{2\pi\Z}};\LR{p_1}(\Omega)}} 
&= \Bp{\int_{\R/{2\pi\Z}} \norm{\iFT_{\R^n}\bb{\snorm{\xi}^{\beta-1}}*_{\R^n}\gamma_\beta *_{{\R/{2\pi\Z}}} F (t,\cdot)}_{p_1}^{r_1}\,\dt}^\frac{1}{r_1}\\
&\leq \Cc{c} \Bp{\int_{\R/{2\pi\Z}} \norm{\gamma_\beta *_{{\R/{2\pi\Z}}} F (t,\cdot)}_{q}^{r_1}\,\dt}^\frac{1}{r_1}\\
&\leq \Cc{c} \Bp{\int_{\R^n} \norm{\gamma_\beta *_{{\R/{2\pi\Z}}} F (x,\cdot)}_{r_1}^q\,\dx}^\frac{1}{q} \leq \Cc{c}\,\norm{F}_q \leq \Cc{c}\,\norm{u}_{1,2,q},
\end{align*}
where Minkowski's integral inequality is employed to conclude the second inequality above.
Classical Sobolev embedding yields $\grad\proj\uvel\in\LR{p_1}(\R^n)$ with $\norm{\grad\proj\uvel}_{p_1}\leq \Cc{c}\norm{\uvel}_{1,2,q}$. By the above,
it thus follows that $\norm{\grad\uvel}_{\LR{r_1}\np{{\R/{2\pi\Z}};\LR{p_1}(\Omega)}} \leq \Cc{c}\norm{\uvel}_{1,2,q}$. By interpolation, the same estimate follows 
for all $r_1,p_1\in[q,\infty)$ satisfying \eqref{SobEmbeddingThm_Condp1}.
In a similar manner, the estimate
$\norm{\uvel}_{\LR{r_0}\np{{\R/{2\pi\Z}};\LR{p_0}(\Omega)}}\leq\Cc{c}\norm{\uvel}_{1,2,q}$ 
can be shown for parameters $r_0,p_0\in[q,\infty)$ satisfying \eqref{SobEmbeddingThm_Condp0}. This concludes the theorem.
\end{proof}

\section{Linearized problem}

A suitable linearization of \eqref{intro_nspastbody} is given by the time-periodic Oseen system
\begin{align}\label{lin_linsystem}
\begin{pdeq}
&\partial_t\uvel  - \nu\Delta\uvel + \rey\partial_1\uvel + \grad\upres =  F && \tin\OmegaR,\\
&\Div\uvel =0 && \tin\OmegaR,\\
&\uvel=0&& \ton\partialOmegaR,\\
&\lim_{\snorm{x}\ra\infty} \uvel(t,x) = 0,\quad \uvel(t+\per,x)=\uvel(t),\quad \upres(t+\per,x)=\upres(t),
\end{pdeq}
\end{align} 
where $\rey>0$.
The goal in this section is to identify a Banach space $X$ of functions $(\uvel,\upres)$ satisfying \eqrefsub{lin_linsystem}{2-4}
such that the differential operator on the left-hand side in \eqrefsub{lin_linsystem}{1} 
\begin{align*}
\linoproseen (\uvel,\upres):=\partial_t\uvel + \rey\partial_1\uvel - \Delta\uvel + \grad\upres
\end{align*}
becomes a homeomorphism
$\linoproseen:X\ra\LRper{q}(\OmegaR)^3$. 
In other words, what can be referred to as ``maximal $\LR{q}$ regularity'' of the time-periodic system \eqref{lin_linsystem} shall 
be established.

The projections $\proj$ and $\projcompl$ shall be used to decompose \eqref{lin_linsystem} into two problems.
More specifically, for data $F\in\LRper{q}\np{\OmegaR}^3$ a solution to \eqref{lin_linsystem} is investigated 
as the sum of a solution corresponding to the steady-state part of the data $\proj F$, and a solution corresponding to the oscillatory part $\projcompl F$. We start with the latter and consider data in the space $\LRcomplper{q}\np{\OmegaR}^3$. 
In this case, appropriate $\LR{q}$ estimates can be established irrespectively of whether $\rey$ vanishes or not. The case $\rey=0$ is therefore included in the theorem below.  

\begin{thm}\label{lin_maxregcomplspacesThm}
Let $\Omega\subset\R^3$ be an exterior domain of class $\CR{2}$, 
$q\in(1,\infty)$ and $\rey\in[0,\rey_0]$.
For any vector field $\F\in\LRcomplper{q}(\OmegaR)^3$ there is a solution 
\begin{align}\label{lin_maxregcomplspacesThmSolSpace}
(\uvel,\upres)\in\WSRcomplper{1,2}{q}(\OmegaR)^3\times\DSRcomplper{1}{q}(\OmegaR)
\end{align}
to \eqref{lin_linsystem} 
which satisfies 
\begin{align}\label{ext_lqestThmEst}
\norm{\uvel}_{1,2,q} + \norm{\grad\upres}_{q} \leq \Cc[ext_lqestThmEstConst]{C} \norm{\F}_{q},
\end{align}
with $\Cclast{C}=\Cclast{C}(q,\Omega,\nu,\rey_0)$. 
If $r\in(1,\infty)$ and $(\tuvel,\tupres)\in\WSRcomplper{1,2}{r}(\OmegaR)^3\times\DSRcomplper{1}{r}(\OmegaR)$ is another 
solution, then $\tuvel=\uvel$ and $\tupres=\upres+d(t)$ for some $\per$-periodic function $d:\R\ra\R$.  
\end{thm}

The proof of Theorem \ref{lin_maxregcomplspacesThm} will be based on three lemmas. The first lemma states that the theorem holds in the case $q=2$.

\begin{lem}\label{lin_maxregcomplspacesq2Lem} 
Let $\Omega$ and $\rey$ be as in Theorem \ref{lin_maxregboundeddomainThm}. For any $F\in\LRcomplper{2}(\OmegaR)^3$ there is a solution $(\uvel,\upres)\in\WSRcomplper{1,2}{2}(\OmegaR)^3\times\DSRcomplper{1}{2}(\OmegaR)$ to \eqref{lin_linsystem}. Moreover, the solution obeys the estimate
\begin{align}\label{lin_maxregcomplspacesq2LemEst}
\norm{\uvel}_{1,2,2} + \norm{\grad\upres}_{2} \leq \Cc{C}\,\norm{F}_{2},
\end{align}
with $\Cclast{C}=\Cclast{C}(\rey_0,\per,\Omega)$
\end{lem}
\begin{proof}
The proof in \cite[Lemma 5]{GaldiTPFlowMovingCylinder} can easily be adapted to establish the desired statement. For the sake of completeness, a sketch of the proof is given here. For a Hilbert space $H$, we introduce the function space $\LRper{2}\bp{\R;H}$ whose elements are the $\per$-periodic
extensions of the functions in $\LR{2}\bp{(0,\per);H}$. Classical theory on Fourier series, in particular the theorem of 
Parseval, is available for such spaces. Since clearly $\LRper{2}(\OmegaR)=\LRper{2}\bp{\R;\LR{2}(\Omega)}$, we 
may express the data $F$ as a Fourier series $F=\sum_{k\in\Z} F_k \e^{i\perf kt}$ with Fourier coefficients 
$F_k\in\LR{2}(\Omega)^3$. Since $\proj F=0$, it follows that $F_0=0$. 
Consider for each $k\in\Z\setminus\set{0}$  the system 
\begin{align*}
\begin{pdeq}
&ik\perf \uvel_k -\nu \Delta \uvel_k + \rey \partial_1 \uvel_k + \grad \upres_k = F_k && \tin\Omega,\\
&\Div \uvel_k = 0 && \tin\Omega,\\
&\uvel_k =0 && \ton \partial\Omega.
\end{pdeq}
\end{align*}
Standard methods from the theory on elliptic systems can be employed to investigate this problem. As a result, existence of 
a solution $(\uvel_k,\upres_k)\in\WSR{2}{2}(\Omega)\times\bp{\DSR{1}{2}(\Omega)\cap\LR{6}(\Omega)}$ which satisfies
\begin{align}\label{lin_maxregcomplspacesq2Lem_FCest}
\perf\snorm{k}\norm{\uvel_k}_2 + \norm{\grad^2\uvel_k}_{2} + \norm{\grad\upres_k}_2 \leq \Cc{c} \, \norm{F_k}_2,
\end{align}
with $\Cclast{c}$ independent on $k$, can be established.
Observe that $\DSR{1}{2}(\Omega)\cap\LR{6}(\Omega)$ is a Hilbert space in the norm $\norm{\grad\cdot}_2$. Thus, by
\eqref{lin_maxregcomplspacesq2Lem_FCest} and Parseval's theorem, the vector fields 
\begin{align*}
\uvel:=\sum_{k\in\Z\setminus\set{0}} \uvel_k\,\e^{i\perf kt},\quad 
\upres:=\sum_{k\in\Z\setminus\set{0}} \upres_k\,\e^{i\perf kt},\quad 
\end{align*}
are well-defined as elements of $\LRper{2}\bp{\R;\WSR{2}{2}(\Omega)}$ and $\LRper{2}\bp{\R;\DSR{1}{2}(\Omega)}$, respectively, 
with $\partial_t\uvel\in\LRper{2}\bp{\R;\LR{2}(\Omega)}$. Parseval's theorem yields
\begin{align*}
\norm{\partial_t\uvel}_{\LRper{2}\np{\R;\LR{2}(\Omega)}}+ \norm{\uvel}_{\LRper{2}\np{\R;\WSR{2}{2}(\Omega)}}+\norm{\grad\upres}_{\LRper{2}\np{\R;\LR{2}(\Omega)}} \leq \Cc{c} \norm{F}_{\LRper{2}\np{\R;\LR{2}(\Omega)}}.
\end{align*}
Finally observe that $\proj\uvel=\proj\upres=0$ as both Fourier coefficients $\uvel_0$ and $\upres_0$ vanish by definition of $\uvel$
and $\upres$. We thus conclude that $(\uvel,\upres)\in\WSRcomplper{1,2}{2}(\OmegaR)^3\times\DSRcomplper{1}{2}(\OmegaR)$ and satisfies
\eqref{lin_maxregcomplspacesq2LemEst}. By construction, $(\uvel,\upres)$ is a solution to \eqref{lin_linsystem}.
\end{proof}

The next lemma states that the assertions in Theorem \ref{lin_maxregcomplspacesThm} are valid if $\Omega$ is replaced with a bounded domain.
\begin{lem}\label{lin_maxregboundeddomainThm}
Let $\boundeddomain\subset\R^3$ be a bounded domain of class $\CR{2}$, 
$q\in(1,\infty)$ and $\rey\in[0,\rey_0]$.
For any vector field $\F\in\LRcomplper{q}(\boundeddomainR)^3$ there is a solution 
\begin{align}\label{lin_maxregboundeddomainThmSolSpace}
(\uvel,\upres)\in\WSRcomplper{1,2}{q}(\boundeddomainR)^3\times\DSRcomplper{1}{q}(\boundeddomainR)
\end{align}
to 
\begin{align}\label{lin_maxregboundeddomainThm_eqboundeddomain}
\begin{pdeq}
&\partial_t\uvel  - \nu\Delta\uvel + \rey\partial_1\uvel + \grad\upres =  F && \tin\boundeddomainR,\\
&\Div\uvel =0 && \tin\boundeddomainR,\\
&\uvel=0&& \ton\partialboundeddomainR,\\
\end{pdeq}
\end{align}  
which satisfies 
\begin{align}\label{lin_maxregboundeddomainThmEst}
\norm{\uvel}_{1,2,q} + \norm{\grad\upres}_{q} \leq \Cc[lin_maxregboundeddomainThmConst]{C} \norm{\F}_{q},
\end{align}
with $\Cclast{C}=\Cclast{C}(q,\boundeddomain,\nu,\rey_0)$. 
If $r\in(1,\infty)$ and $(\tuvel,\tupres)\in\WSRcomplper{1,2}{r}(\boundeddomainR)^3\times\DSRcomplper{1}{r}(\boundeddomainR)$ is another 
solution, then $\tuvel=\uvel$ and $\tupres=\upres+d(t)$ for some $\per$-periodic function $d:\R\ra\R$.  
\end{lem}

\begin{proof}
One may verify that the proof in \cite[Lemma 9]{GaldiTPFlowMovingCylinder} for a two-dimensional domain also holds for a three-dimensional
domain. For the sake of completeness, a sketch is given here. 
By density of $\CRcicomplper(\boundeddomain)$ in $\LRcomplper{q}\np{\boundeddomainR}$, it 
suffices to consider only $\F\in\CRcicomplper(\boundeddomain)^3$.
Starting point is a solution 
$(\uvel,\upres)\in\WSRcomplper{1,2}{2}(\boundeddomainR)^3\times\DSRcomplper{1}{2}(\boundeddomainR)$, the existence of which can be shown
by the same methods as in the proof of Lemma \ref{lin_maxregcomplspacesq2Lem}.   
By Sobolev's embedding theorem, it may be assumed that this solution is continuous in the sense that 
$\uvel\in\CRper\bp{\R;\LRsigma{2}(\boundeddomain)}$. Clearly, for this solution a $t_0\in[0,\per]$ can be chosen such that $\uvel(t_0)\in\WSR{2}{2}(\boundeddomain)$ 
and $\Div\uvel(t_0)=0$. By Sobolev's embedding theorem, it follows that $\uvel(t_0)\in\LRsigma{q}(\boundeddomain)$. 
Consider now the initial-value problem 
\begin{align}\label{bd_galdirepZeroRHSIVP}
\begin{pdeq}
&\partial_t\vvel  = \nu\Delta\vvel - \grad\vpres  && \tin(t_0,\infty)\times\boundeddomain,\\
&\Div\vvel =0 && \tin(t_0,\infty)\times\boundeddomain,\\
&\vvel=0&& \ton(t_0,\infty)\times\partial\boundeddomain,\\
&\vvel(t_0,\cdot)=\uvel(t_0,\cdot) && \tin\boundeddomain.
\end{pdeq}
\end{align} 
It is well known that the Stokes operator $A:=\hproj\Delta$ generates a
bounded analytic semi-group on $\LRsigma{q}(\boundeddomain)$; see \cite{Giga_StokesSemigroupAna81}. Consequently, the solution to
the initial-value problem \eqref{bd_galdirepZeroRHSIVP}
given by $\vvel:=\exp\bp{A(t-t_0)}\uvel(t_0)$ satisfies
\begin{align}\label{bd_lqestThm_LqEstvvel}
\forall t>t_0:\quad \norm{\partial_t\vvel(t)}_q + \norm{A\vvel(t)}_q  \leq \Cc{c}\, (t-t_0)^{-1}; 
\end{align}
see for example \cite[Theorem II.4.6]{NagelEngelBook}.
Also by classical results, see for example \cite[Theorem 2.8]{GigaSohrLpEst91}, there exists a solution 
$(\wvel,\wpres)\in\WSR{1,2}{q}\bp{(t_0,\infty)\times\boundeddomain}\times\WSR{0,1}{q}\bp{(t_0,\infty)\times\boundeddomain}$ 
to 
\begin{align}\label{bd_galdirepZeroIVIVP}
\begin{pdeq}
&\partial_t\wvel  = \nu\Delta\wvel - \grad\wpres + \F && \tin(t_0,\infty)\times\boundeddomain,\\
&\Div\wvel =0 && \tin(t_0,\infty)\times\boundeddomain,\\
&\wvel=0&& \ton(t_0,\infty)\times\partial\boundeddomain,\\
&\wvel(t_0,\cdot)=0 && \tin\boundeddomain
\end{pdeq}
\end{align}  
that is continuous in the sense $\wvel\in\CRper{}\bp{[t_0,\infty);\LRsigma{q}(\boundeddomain)}$ and satisfies
\begin{align}\label{bd_lqestThm_LqEstwvel}
\forall \tau\in(t_0,\infty):\quad \norm{\wvel}_{\WSR{1,2}{q}\np{(t_0,\tau)\times\boundeddomain}} + \norm{\wpres}_{\WSR{0,1}{q}\np{(t_0,\tau)\times\boundeddomain}}  \leq \Cc[bd_lqestThm_LqEstwvelConst]{c}\, 
\norm{\F}_{\LR{q}\np{(t_0,\tau)\times\boundeddomain}}
\end{align}
with $\const{bd_lqestThm_LqEstwvelConst}$ \emph{independent} on $\tau$. Since $\uvel$ and $\vvel+\wvel$ solve the same initial-value problem, a standard uniqueness
argument implies $\uvel=\vvel+\wvel$. Due to the $\per$-periodicity of $\uvel$ and $\F$, it follows for all $m\in\N$ that 
\begin{align}\label{bd_lqestThm_LqEstGaldiTrick}
\begin{aligned}
\int_0^{\per} \norm{\partial_t\uvel(t)}_q^q + \norm{A\uvel(t)}_q^q\,\dt &=
\frac{1}{m}\int_{2\per}^{\np{m+2}\per} \norm{\partial_t\uvel(t)}_q^q + \norm{A\uvel(t)}_q^q\,\dt\\
&\leq \Cc{c}\frac{1}{m}\int_\per^\infty t^{-q}\,\dt+  
\Cc[bd_lqestThmTmpConst1]{c}\frac{1}{m}\norm{\F}^q_{\LR{q}\np{(0,(m+1)T)\times\boundeddomain}}\\
&\leq \Cc{c}\frac{1}{m} \frac{1}{q-1} \per^{1-q}+  
\const{bd_lqestThmTmpConst1} \frac{m+1}{m}\norm{\F}^q_{\LR{q}\np{(0,T)\times\boundeddomain}}.
\end{aligned}
\end{align}
Now let $m\ra\infty$ to conclude
\begin{align*}
\norm{\partial_t\uvel}_{\LRper{q}(\boundeddomainR)} + \norm{A\uvel}_{\LRper{q}(\boundeddomainR)} \leq \Cc{c} \norm{\F}_{\LRper{q}(\boundeddomainR)}.
\end{align*}
The estimate $\norm{\grad^2\uvel}_{\LRper{q}(\boundeddomainR)}\leq \Cc{c}\norm{A\uvel}_{\LRper{q}(\boundeddomainR)}$ is a consequence of well-known $\LR{q}$ theory for the Stokes problem in bounded domains; see for example \cite[Theorem IV.6.1]{GaldiBookNew}. 
Consequently, the estimate  
$\norm{\uvel}_{1,2,q} \leq \Cc{c} \norm{\F}_{q}$ follows by  
employing Poincar\'{e}'s inequality $\norm{\uvel}_q\leq \Cc{c}\,\norm{\partial_t{\uvel}}_q$. 
Now modify the pressure $\upres$ by adding a function depending only on $t$ such that $\int_\boundeddomain \upres(t,x)\,\dx=0$, which
ensures the validity of Poincar\'{e}'s inequality for $\upres$.
A similar estimate is then obtained for $\upres$ by
isolating $\grad\upres$ in \eqrefsub{lin_maxregboundeddomainThm_eqboundeddomain}{1}. This establishes \eqref{lin_maxregboundeddomainThmEst}.
To show the statement of uniqueness, a duality argument can be employed. 
For this purpose, let $\testphi\in\CRciper(\boundeddomainR)$ and let $(\adjvel,\adjpres)$ be a solution to the problem
\begin{align}\label{bd_linnsadj}
\begin{pdeq}
&\partial_t\adjvel  = -\nu\Delta\adjvel - \grad\upres + \testphi && \tin\boundeddomainR,\\
&\Div\adjvel =0 && \tin\boundeddomainR,\\
&\adjvel=0&& \ton\partialboundeddomainR,\\
&\adjvel(t+\per,x)=\adjvel(t,x)
\end{pdeq}
\end{align} 
adjoint to \eqref{lin_maxregboundeddomainThm_eqboundeddomain}. The existence of a solution $(\adjvel,\adjpres)$ follows by the same arguments from above that yield a solution to \eqref{lin_maxregboundeddomainThm_eqboundeddomain}. 
Since $\testphi\in\CRciper(\boundeddomainR)$, the solution satisfies $(\adjvel,\adjpres)\in\WSRper{1,2}{s}(\boundeddomainR)\times\WSRper{1}{s}(\boundeddomainR)$ for all $s\in(1,\infty)$.
The regularity of $(\adjvel,\adjpres)$ ensures validity of the following computation:
\begin{align*}
\int_0^\per\int_\boundeddomain (\wvel-\twvel)\cdot \testphi\,\dx\dt 
&= \int_0^\per\int_\boundeddomain (\wvel-\twvel)\cdot (\partial_t\adjvel +\nu\Delta\adjvel +\grad\upres)\,\dx\dt  \\
&= \int_0^\per\int_\boundeddomain \bp{\partial_t\nb{\wvel-\twvel}-\nu\Delta\nb{\wvel-\twvel}+\grad\nb{\wpres-\twpres}}\cdot \adjvel\,\dx\dt=0.
\end{align*}
Since $\testphi\in\CRciper(\boundeddomainR)$ was arbitrary, $\twvel-\wvel=0$ follows. In turn, $\grad\wpres=\grad\twpres$ and thus $\twpres=\wpres+d(t)$ follows.
\end{proof}

The final lemma concerns estimates of the pressure term in \eqref{lin_linsystem}.
The following lemma was originally
proved for a two-dimensional exterior domain in \cite{GaldiTPFlowMovingCylinder}. We employ the ideas from \cite[Proof of Lemma 6]{GaldiTPFlowMovingCylinder} and establish the lemma in a slightly modified form for a three-dimensional exterior domain.

\begin{lem}\label{lin_estofpressureLem}
Let $\Omega$ and $\rey$ be as in Theorem \ref{lin_maxregboundeddomainThm} and $s\in(1,\infty)$. There is a constant $\Cc[lin_estofpressureLemConstPresEst]{C}=\Cclast{C}(\Rzero,\Omega,s)$ 
such that a solution 
$(\uvel,\upres)\in\WSRcomplper{1,2}{r}(\OmegaR)^3\times\DSRcomplperRzero{1}{r}(\OmegaR)$ to \eqref{lin_linsystem} corresponding
to data $F\in\LRcomplper{r}(\OmegaR)^3$ for some $r\in(1,\infty)$ satisfies for a.e. $t\in\R$:
\begin{align}
&\begin{aligned}
\norm{\upres(t,\cdot)}_{\frac{3}{2}s,\Omega_\Rzero} \leq \const{lin_estofpressureLemConstPresEst}\,\Bp{
\norm{F(t,\cdot)}_s+
\norm{\grad\uvel(t,\cdot)}_{s,\Omega_\Rzero} + 
\norm{\grad\uvel(t,\cdot)}_{s,\Omega_\Rzero}^{\frac{s-1}{s}}\,\norm{\grad\uvel(t,\cdot)}_{1,s,\Omega_\Rzero}^\frac{1}{s}
}
\end{aligned}\label{lin_estofpressureLem_est1}
\end{align}
Moreover, for every $\rho>\Rstar$ there is a constant $\Cc[lin_estofpressureLemConstGradPresEst]{C}=\Cclast{C}(\rho,\Omega,s)$ such that
for a.e. $t\in\R$:
\begin{align}
&\begin{aligned}
\norm{\grad\upres(t,\cdot)}_{s,\Omega^\rho} \leq \const{lin_estofpressureLemConstGradPresEst}\,\bp{\norm{F(t,\cdot)}_{s} + \norm{\upres(t,\cdot)}_{s,\Omega_\rho} }.
\end{aligned}\label{lin_estofpressureLem_est2}
\end{align}
 
\end{lem}
\begin{proof}
For the sake of simplicity, the $t$-dependency of functions is not indicated. All norms are taken with respect to the spatial variables only.
Consider an arbitrary $\phi\in\CRci\np{\overline{\Omega}}$.
Observe that for any $\psi\in\CRiper(\R)$ holds
\begin{align*}
\int_{0}^\per\int_{\Omega}\partial_t\uvel\cdot\grad\phi\,\dx\, \psi\,\dt = 
\int_{0}^\per\int_{\Omega}\Div\uvel\, \phi\, \partial_t\psi\,\dx\dt =0,
\end{align*} 
which implies $\int_{\Omega}\partial_t\uvel\cdot\grad\phi\,\dx=0$ for a.e. $t$. Moreover, 
\begin{align*}
\int_\Omega \partial_1\uvel\cdot\grad\phi\,\dx = -\int_\Omega \Div\uvel \cdot\partial_1\phi\,\dx =  0. 
\end{align*}
Hence it follows from \eqref{lin_linsystem} that $\upres$ is a solution to the weak Neumann problem for the Laplacian:
\begin{align*}
\forall\phi\in\CRci\np{\overline{\Omega}}:\ \int_\Omega\grad\upres\cdot\grad\phi\,\dx = \int_\Omega F\cdot \grad\phi + \Delta\uvel\cdot\grad\phi\,\dx.
\end{align*}
Recall that
\begin{align}\label{lin_estofpressureLem_NeumannproblemClass}
\upres\in\LRloc{1}(\Omega)\quad\wedge\quad\grad\upres\in\LR{r}(\Omega)^3\quad\wedge\quad\int_{\Omega_{\Rzero}} \upres\,\dx=0.
\end{align}
It is well-known that the weak Neumann problem for the Laplacian in an exterior domain is uniquely solvable in the class \eqref{lin_estofpressureLem_NeumannproblemClass}; 
see for example \cite[Section III.1]{GaldiBookNew} or \cite{SimaderWeakDirichletNeumann1990}.
We can thus write $\upres$ as a sum $\upres = \upres_1+\upres_2$ of two solutions (in the class above) to the weak Neumann problem
\begin{align*}
\forall\phi\in\CRci\np{\overline{\Omega}}:\ \int_\Omega\grad\upres_1\cdot\grad\phi\,\dx = \int_\Omega F\cdot \grad\phi \,\dx
\end{align*}
and 
\begin{align*}
\forall\phi\in\CRci\np{\overline{\Omega}}:\ \int_\Omega\grad\upres_2\cdot\grad\phi\,\dx = \int_\Omega  \Delta\uvel\cdot\grad\phi\,\dx,
\end{align*}
respectively. The \textit{a priori} estimate 
\begin{align}\label{lin_estofpressureLem_NeumannproblemAprioriEst1}
&\forall q\in(1,\infty):\ \norm{\grad\upres_1}_q\leq \Cc{c}\,\norm{F}_q
\end{align}
is well-known.
An estimate of $\upres_2$ shall now be established. 
Consider for this purpose an arbitrary function
$g\in\CRci(\Omega_\Rzero)$ with $\int_{\Omega_\Rzero} g\,\dx = 0$. Existence of a vector field $h\in\CRci(\Omega_\Rzero)$ with $\Div h = g$ and 
\begin{align*}
\forall q\in(1,\infty):\ \norm{h}_{1,q} \leq \Cc{c}\,\norm{g}_q
\end{align*}
is well-known; see for example \cite[Theorem III.3.3]{GaldiBookNew}. Let $\Phi$ be a solution to the following weak Neumann problem for the Laplacian:
\begin{align*}
\forall\phi\in\CRci\np{\overline{\Omega}}:\ \int_\Omega\grad\Phi\cdot\grad\phi\,\dx = \int_\Omega  h\cdot\grad\phi\,\dx.
\end{align*} 
By classical theory, such a solution exists with
\begin{align*}
\forall q\in(1,\infty):\ \Phi\in\CRi(\overline{\Omega})\ \wedge\ \norm{\grad\Phi}_{1,q} \leq \Cc{c}\,\norm{h}_{1,q} \leq \Cc{c}\,\norm{g}_{q}.
\end{align*}
Since $\Phi$ is harmonic in $\R^3\setminus\overline{\B_\Rzero}$, the following asymptotic 
expansion as $\snorm{x}\ra\infty$ is valid: 
\begin{align*}
\partial^\alpha \Phi(x) = \partial^\alpha \Cc{c} + \partial^\alpha\fundsollaplace(x)\cdot \int_{\partial\B_{\rho}} \pdn{\Phi}\,\dS + \bigo\bp{\snorm{x}^{-2-\snorm{\alpha}}},
\end{align*}
where $\Cclast{c}$ is a constant and $\fundsollaplace:\R^3\setminus\set{0}\ra\R,\ \fundsollaplace(x):=(4\pi \snorm{x})^{-1}$ the fundamental solution of the Laplacian in $\R^3$.
Observing that 
\begin{align*}
\int_{\partial\B_{\rho}} \pdn{\Phi}\,\dS = \int_{\partial\Omega} \pdn{\Phi}\,\dS + \int_{\Omega_\rho} \Delta\Phi\,\dx = 
0 + \int_{\Omega_\rho} \Div h \,\dx =
\int_{\Omega_\rho} g\,\dx = 0,  
\end{align*}
we thus deduce $\grad\Phi(x) = \bigo\bp{\snorm{x}^{-3}}$  as $\snorm{x}\ra\infty$.
Similarly, we see that $\upres_2 = \Cc{c}+\bigo\bp{\snorm{x}^{-1}}$. We therefore conclude that
\begin{align*}
\lim_{\rho\ra\infty}\int_{\partial\B_\rho} \upres_2\, \pdn{\Phi}\,\dx = 0.
\end{align*}
We can thus compute
\begin{align}\label{lin_estofpressureLem_Comppres2}
\begin{aligned}
\int_\Omega \upres_2\, g\,\dx &= \int_\Omega \upres_2\, \Delta\Phi\,\dx= \lim_{\rho\ra\infty}\int_{\Omega_\rho} \upres_2\, \Delta\Phi\,\dx\\
&= \lim_{\rho\ra\infty} \Bp{\int_{\partial\Omega_\rho}\upres_2\, \pdn{\Phi}\,\dS - \int_{\Omega_\rho}\grad\upres_2\cdot\grad\Phi\,\dx}\\
&= - \int_{\Omega}\grad\upres_2\cdot\grad\Phi\,\dx 
= - \int_{\Omega}\Delta\uvel\cdot\grad\Phi\,\dx. 
\end{aligned}
\end{align}
The decay of $\grad\Phi$ further implies $\partial_i\uvel_j\,\partial_k\Phi\in\LR{1}(\B^{2\Rzero})$. We can thus find 
a sequence $\seqkN{\rho}$ of positive numbers with $\lim_{k\ra\infty}\rho_k=\infty$ such that 
\begin{align*}
\lim_{\rho_k\ra\infty}\int_{\partial\B_\rho} \partial_j\uvel_i\,\partial_i\Phi\, n_j - \partial_j\uvel_i\,\partial_j\Phi\, n_i \,\dS = 0.
\end{align*}
Returning to \eqref{lin_estofpressureLem_Comppres2}, we continue the computation and find that
\begin{align*}
\int_\Omega \upres_2\, g\,\dx &= - \lim_{k\ra\infty} \int_{\Omega_{\rho_k}}\Delta\uvel\cdot\grad\Phi\,\dx\\
&= - \lim_{k\ra\infty} {\int_{\partial\Omega_{\rho_k}}\partial_j\uvel_i\,\partial_i\Phi\, n_j - \partial_j\uvel_i\,\partial_j\Phi\, n_i \dS} \\
&= - {\int_{\partial\Omega}\partial_j\uvel_i\,\partial_i\Phi\, n_j - \partial_j\uvel_i\,\partial_j\Phi\, n_i \dS}
= - {\int_{\partial\Omega}\grad\uvel : \bp{\grad\Phi\otimes n - n\otimes\grad\Phi} \dS}.
\end{align*}
Applying first the H\"older and then a trace inequality (for example \cite[Theorem II.4.1]{GaldiBookNew}), we deduce 
\begin{align*}
\snormL{\int_\Omega \upres_2\, g\,\dx} &\leq \Cc{c}\,\norm{\grad\uvel}_{s,\partial\Omega}\,\norm{\grad\Phi}_{\frac{s}{s-1},\partial\Omega}\\
&\leq \Cc{c}\,\norm{\grad\uvel}_{s,\partial\Omega}\,\norm{\grad\Phi}_{1,\frac{3s}{3s-2},\Omega_\Rzero} 
\leq \Cc{c}\,\norm{\grad\uvel}_{s,\partial\Omega}\,\norm{g}_{\frac{3s}{3s-2},\Omega_\Rzero}.
\end{align*}
Recalling that $\int_{\Omega_\Rzero}\upres_2\,\dx=0$, we thus obtain
\begin{align*}
\norm{\upres_2}_{\frac{3}{2}s,\Omega_\Rzero}
 = \sup_{\overset{g\in\CRci(\Omega_\Rzero), \norm{g}_\frac{3s}{3s-2}=1}{\int_{\Omega_\Rzero}g\,\dx=0}}
\snormL{\int_\Omega \upres_2\, g\,\dx} \leq \Cc{c}\,\norm{\grad\uvel}_{s,\partial\Omega}.
\end{align*}
Another application of a trace inequality (see again \cite[Theorem II.4.1]{GaldiBookNew}) now yields
\begin{align*}
\norm{\upres_2}_{\frac{3}{2}s,\Omega_\Rzero} &\leq \Cc{c}\,\Bp{
\norm{\grad\uvel}_{s,\Omega_\Rzero} + 
\norm{\grad\uvel}_{s,\Omega_\Rzero}^{\frac{s-1}{s}}\,\norm{\grad\uvel}_{1,s,\Omega_\Rzero}^\frac{1}{s}
}.
\end{align*}
Recalling \eqref{lin_estofpressureLem_NeumannproblemAprioriEst1}, we employ Sobolev's embedding theorem to finally conclude
\begin{align*}
\norm{\upres}_{\frac{3}{2}s,\Omega_\Rzero} &\leq \Cc{c}\,\norm{\grad\upres_1}_{s,\Omega_\Rzero} +  \norm{\upres_2}_{\frac{3}{2}s,\Omega_\Rzero}\\
&\leq \Cc{c}\,\Bp{
\norm{F}_s+
\norm{\grad\uvel}_{s,\Omega_\Rzero} + 
\norm{\grad\uvel}_{s,\Omega_\Rzero}^{\frac{s-1}{s}}\,\norm{\grad\uvel}_{1,s,\Omega_\Rzero}^\frac{1}{s}
}
\end{align*}
and thus \eqref{lin_estofpressureLem_est1}. 
To show \eqref{lin_estofpressureLem_est2}, we introduce $\Rmiddel\in(\Rstar,\rho)$ and a ``cut-off'' function $\cutoff\in\CRi(\R^3;\R)$
with $\cutoff=1$ on $\Omega^\Rmiddel$ and $\cutoff=0$ on $\Omega_\Rstar$. We then put $\wpres:=\cutoff\upres$ and observe from \eqref{lin_linsystem}
that $\wpres$ is a solution to the weak Neumann problem for the Laplacian
\begin{align}\label{lin_estofpressureLem_weakNeumannAfterCutoff}
\forall\phi\in\CRci\np{\overline{\Omega}}:\ \int_\Omega \grad\wpres\cdot\grad\phi\,\dx = \linf{\calf_1}{\phi} + \linf{\calf_2}{\phi}
\end{align}
with
\begin{align*}
\linf{\calf_1}{\phi} := \int_{\Omega} \bp{2\upres\,\grad\cutoff + \cutoff F}\cdot\grad\phi\,\dx,\quad 
\linf{\calf_2}{\phi} := \int_{\Omega} \bp{\grad\cutoff \cdot F+\Delta\cutoff\,\upres }\,\phi\,\dx. 
\end{align*}
We clearly have 
\begin{align*}
\sup_{\norm{\grad\phi}_{s^*}=1}\snorm{\linf{\calf_1}{\phi}} \leq \Cc{c}\,\bp{\norm{\upres}_{s,\Omega_\rho} + \norm{\calf}_s}.
\end{align*}
Since $\cutoff=1$ on $\Omega^\Rmiddel$, we further observe that 
\begin{align*}
\sup_{\norm{\grad\phi}_{s^*}=1}\snorm{\linf{\calf_2}{\phi}} = \sup_{\overset{\norm{\grad\phi}_{s^*}=1}{\supp\phi\subset{\Omega_\rho}}}\snorm{\linf{\calf_2}{\phi}}
\leq \Cc{c}\,\bp{\norm{F}_{s} + \norm{\upres}_{s,\Omega_\rho}},
\end{align*}
where Poincar\'{e}'s inequality is used to obtain the last estimate. A standard \textit{a priori} estimate for the weak Neumann problem \eqref{lin_estofpressureLem_weakNeumannAfterCutoff} now implies \eqref{lin_estofpressureLem_est2}.
\end{proof}

\begin{proof}[Proof of Theorem \ref{lin_maxregcomplspacesThm}]
By density of $\CRcicomplper(\OmegaR)$ in $\LRcomplper{q}\np{\OmegaR}$, it 
suffices to consider only $\F\in\CRcicomplper(\OmegaR)^3$. The starting point will be the solution 
$(\uvel,\upres)\in{\WSRcomplper{1,2}{2}(\OmegaR)^3\times\DSRcomplper{1}{2}}(\OmegaR)$
from Lemma \ref{lin_maxregcomplspacesq2Lem}. By adding to $\upres$ a function that only depends on time, we may assume without loss of generality that $\int_{\Omega_{\Rzero}} \upres\,\dx=0$.
For the scope of the proof, we fix a constant $\rhoover$ with $\Rstar<\rhounder<\Rzero$.

We shall establish two fundamental estimates. To show the first one, we introduce a cut-off function 
$\cutoffunder\in\CRi(\R^3;\R)$ with $\cutoffunder(x)=1$ for $\snorm{x}\geq\rhounder$ and $\cutoffunder(x)=0$ for $\snorm{x}\leq\Rstar$. 
We let $\fundsollaplace:\R^3\setminus\set{0}\ra\R,\ \fundsollaplace(x):=(4\pi \snorm{x})^{-1}$ denote the fundamental solution to the Laplace operator and put
\begin{align}\label{lin_maxregcomplspacesThm_defofw}
\begin{aligned}
&\Vvel:\R\times\R^3\ra\R^3,\quad \Vvel=\grad \fundsollaplace*_{\R^3}\bp{\grad\cutoffunder\cdot\uvel} , \\
&\Vpres:\R\times\R^3\ra\R,\quad \Vpres= \fundsollaplace*_{\R^3}\bp{\nb{\partial_t-\Delta+\rey\partial_1}\np{\grad\cutoffunder\cdot\uvel}} , \\
&\wvel:\R\times\R^3\ra\R^3,\quad \wvel(t,x):=\cutoffunder(x)\,\uvel(t,x)- \Vvel(t,x), \\
&\wpres:\R\times\R^3\ra\R,\quad \wpres(t,x):=\cutoffunder(x)\,\upres(t,x)- \Vpres(t,x). 
\end{aligned}
\end{align}
Then $(\wvel,\wpres)$ is a solution to the whole-space problem
\begin{align}\label{lin_linsystemwholespace}
\begin{pdeq}
&\partial_t\wvel - \Delta\wvel + \rey\partial_1\wvel  + \grad\wpres = && \\ 
&\qquad\qquad\qquad
\cutoffunder F - 2\grad\cutoffunder\cdot\grad\uvel-\Delta\cutoffunder \uvel + \rey\partial_1\cutoffunder\uvel + \grad\cutoffunder\upres&& \tin\RthreeR,\\
&\Div\wvel =0 && \tin\RthreeR.
\end{pdeq}
\end{align} 
The precise regularity of $(\wvel,\wpres)$ is not important at this point. It is enough to observe that $\wvel$ and $\wpres$ belong to the space of tempered time-periodic distributions $\TDRper(\RthreeR)$, which is easy to verify from the definition \eqref{lin_maxregcomplspacesThm_defofw} and the regularity of 
$\uvel$ and $\upres$. It is not difficult to show, see \cite[Lemma 5.3]{KyedExRegTPNS2014}, that a solution $\wvel$ to \eqref{lin_linsystemwholespace}
is unique in the class of distributions in $\TDRper(\RthreeR)$ satisfying $\proj\wvel=0$. Consequently, $\wvel$ coincides with 
the solution from \cite[Theorem 2.1]{KyedMaxregTPNS2014} and therefore satisfies 
\begin{align*}
\norm{\wvel}_{1,2,s} 
&\leq \Cc{c}\,\norm{\cutoffunder F - 2\grad\cutoffunder\cdot\grad\uvel-\Delta\cutoffunder \uvel + \rey\partial_1\cutoffunder\uvel + \grad\cutoffunder\upres}_s\\
&\leq \Cc{c}\,\bp{\norm{F}_s + \norm{\uvel}_{s,\torus\times\Omega_\rhounder}+\norm{\grad\uvel}_{s,\torus\times\Omega_\rhounder}+\norm{\upres}_{s,\torus\times\Omega_\rhounder}}
\end{align*}
for all $s\in(1,\infty)$. Clearly, 
\begin{align*}
\norm{\grad\Vvel}_s+\norm{\grad^2\Vvel}_s\leq \Cc{c}\bp{\norm{\uvel}_{s,\Omega_\rhounder}+\norm{\grad\uvel}_{s,\Omega_\rhounder}}.
\end{align*}
Since $\uvel=\wvel+\Vvel$ for  $x\in\Omega^\rhounder$, we conclude
\begin{align*}
&\norm{\grad\uvel}_{s,\torus\times\Omega^\rhounder} + \norm{\grad^2\uvel}_{s,\torus\times\Omega^\rhounder} \\
&\qquad\qquad \leq \Cc{c}\,\bp{\norm{F}_s + 
\norm{\uvel}_{s,\torus\times\Omega_\rhounder}+\norm{\grad\uvel}_{s,\torus\times\Omega_\rhounder}+\norm{\upres}_{s,\torus\times\Omega_\rhounder}}
\end{align*}
for all $s\in(1,\infty)$.
For a similar estimate of $\uvel$ and $\partial_t\uvel$, we turn first to \eqref{lin_linsystem} and then apply \eqref{lin_estofpressureLem_est2} to deduce 
\begin{align*}
\norm{\partial_t\uvel}_{s,\torus\times\Omega^\rhounder} 
&\leq \Cc{c}\,\bp{ \norm{F}_s + \norm{\Delta\uvel}_{s,\torus\times\Omega^\rhounder} + \norm{\rey\partial_1\uvel}_{s,\torus\times\Omega^\rhounder}+\norm{\grad\upres}_{s,\torus\times\Omega^\rhounder}}\\
&\leq \Cc{c}\,\bp{\norm{F}_s + 
\norm{\uvel}_{s,\torus\times\Omega_\rhounder}+\norm{\grad\uvel}_{s,\torus\times\Omega_\rhounder}+\norm{\upres}_{s,\torus\times\Omega_\rhounder}}.
\end{align*}
Since $\proj\uvel=0$, Poincar\'{e}'s inequality yields $\norm{\uvel}_{s,\torus\times\Omega^\rhounder}\leq\Cc{c}\norm{\partial_t\uvel}_{s,\torus\times\Omega^\rhounder}$.
We have thus shown
\begin{align}\label{lin_maxregcomplspacesThm_aprioriEstover}
&\norm{\uvel}_{1,2,s,\torus\times\Omega^\rhounder} \leq \Cc{c}\,\bp{\norm{F}_s + 
\norm{\uvel}_{s,\torus\times\Omega_\rhounder}+\norm{\grad\uvel}_{s,\torus\times\Omega_\rhounder}+\norm{\upres}_{s,\torus\times\Omega_\rhounder}}
\end{align}
for all $s\in(1,\infty)$.

Next, we seek to establish a similar estimate for $\uvel$ over the bounded domain $\torus\times\Omega_\rhounder$.
For this purpose, we introduce a ``cut-off'' function $\cutoffover\in\CRi(\R^3;\R)$ with $\cutoffover(x)=1$ for $\snorm{x}\leq\rhoover$ and 
$\cutoffover(x)=0$ for $\snorm{x}\geq\Rzero$. We then introduce a vector field $\Vvel$ with
\begin{align}\label{lin_maxregcomplspacesThm_bogliftingvecfield}
\begin{aligned}
&\Vvel\in\WSRcomplper{1,2}{2}\np{\R\times\R^3},\quad \supp\Vvel\subset\R\times\Omega_{\rhoover,\Rzero},\quad 
\Div\Vvel = \grad\cutoffover\cdot\uvel,\\
&\forall s\in(1,\infty):\ \norm{\Vvel}_{1,2,s}\leq\Cc{c}\,
\bp{\norm{\uvel}_{s,\torus\times\Omega_{\rhoover,\Rzero}}+\norm{\grad\uvel}_{s,\torus\times\Omega_{\rhoover,\Rzero}} + \norm{\partial_t\uvel}_{s,\torus\times\Omega_{\rhoover,\Rzero}}}.
\end{aligned}
\end{align}
Since 
\begin{align*}
\int_{\Omega_{\rhoover,\Rzero}}\grad\cutoffover\cdot\uvel\,\dx = \int_{\Omega_{\Rzero}}\Div\bp{\cutoffover\uvel}\,\dx =
\int_{\partial\Omega_{\Rzero}} \uvel\cdot n\,\dS = 0, 
\end{align*}
the existence of a vector field $V$ with the properties above can be established by the same construction as the one used in 
\cite[Theorem III.3.3]{GaldiBookNew}; see also \cite[Proof of Lemma 3.2.1]{habil}. We now let
\begin{align}\label{lin_maxregcomplspacesThm_defofw2}
\begin{aligned}
&\wvel:\R\times\R^3\ra\R^3,\quad \wvel(t,x):=\cutoffover(x)\,\uvel(t,x)- \Vvel(t,x), \\
&\wpres:\R\times\R^3\ra\R,\quad \wpres(t,x):=\cutoffover(x)\,\upres(t,x). 
\end{aligned}
\end{align}
Then $(\wvel,\wpres)\in\WSRcomplper{1,2}{2}(\R\times\Omega_\Rzero)\times\DSRcomplper{1}{2}(\R\times\Omega_\Rzero)$ is a solution to the problem
\begin{align*}
\begin{pdeq}
&\partial_t\wvel - \Delta\wvel + \rey\partial_1\wvel  + \grad\wpres  && \\ 
&\quad= \cutoffover F - 2\grad\cutoffover\cdot\grad\uvel-\Delta\cutoffover \uvel + \rey\partial_1\cutoffover\uvel + \grad\cutoffover\upres
+\nb{\partial_t-\Delta+\rey\partial_1}\Vvel 
&& \tin\R\times\Omega_\Rzero,\\
&\Div\wvel =0 && \tin\R\times\Omega_\Rzero,\\
&\wvel = 0 && \ton\R\times\partial\Omega_\Rzero.
\end{pdeq}
\end{align*} 
By Lemma \ref{lin_maxregboundeddomainThm}, we thus deduce  
\begin{align*}
\norm{\wvel}_{1,2,s} 
&\leq \Cc{c}\,\norm{\cutoffover F - 2\grad\cutoffover\cdot\grad\uvel-\Delta\cutoffover \uvel + \rey\partial_1\cutoffover\uvel + \grad\cutoffover\upres+\nb{\partial_t-\Delta+\rey\partial_1}\Vvel}_s\\
&\leq \Cc{c}\,\bp{\norm{F}_s + \norm{\uvel}_{s,\torus\times\Omega_{\rhoover,\Rzero}}+\norm{\grad\uvel}_{s,\torus\times\Omega_{\rhoover,\Rzero}}+\norm{\upres}_{s,\torus\times\Omega_{\rhoover,\Rzero}}+\norm{\Vvel}_{1,2,s}}
\end{align*}
for all $s\in(1,\infty)$.
Since $\Omega_{\rhoover,\Rzero}\subset\Omega^\rhounder$, we can combine \eqref{lin_maxregcomplspacesThm_bogliftingvecfield} and \eqref{lin_maxregcomplspacesThm_aprioriEstover} to estimate
\begin{align*}
\norm{\uvel}_{1,2,s,\Omega_{\rhoover}}
\leq\Cc{c}\,\bp{\norm{F}_s + 
\norm{\uvel}_{s,\torus\times\Omega_\rhounder}+\norm{\grad\uvel}_{s,\torus\times\Omega_\rhounder}+\norm{\upres}_{s,\torus\times\Omega_\rhounder}},
\end{align*}
which was the intermediate goal at this stage. Combining the estimate above with \eqref{lin_maxregcomplspacesThm_aprioriEstover}, we have
\begin{align}\label{lin_maxregcomplspacesThm_aprioriEstfinal}
\norm{\uvel}_{1,2,s}
\leq\Cc{c}\,\bp{\norm{F}_s + 
\norm{\uvel}_{s,\torus\times\Omega_\Rzero}+\norm{\grad\uvel}_{s,\torus\times\Omega_\Rzero}+\norm{\upres}_{s,\torus\times\Omega_\Rzero}}
\end{align} 
for all $s\in(1,\infty)$.

We now move on to the final part of the proof. We emphasize that estimate \eqref{lin_maxregcomplspacesThm_aprioriEstfinal}
has been established for all $s\in(1,\infty)$, but we do not actually know whether the right-hand side is finite or not. At the outset,
we only know the right-hand side is finite for $s\leq 2$. We shall now use a boot-strap argument to show that is also the case for $s\in(2,\infty)$.     
For this purpose, we need the embedding properties of $\WSRper{1,2}{s}(\OmegaR)$ stated in Theorem \ref{SobEmbeddingThm}.
Choosing for example $\alpha=\beta=\frac{1}{2}$ in Theorem \ref{SobEmbeddingThm}, we obtain the implication
\begin{align}\label{lin_maxregcomplspacesThm_bsregvel}
\forall s\in[2,\infty):\quad \uvel\in\WSRper{1,2}{s}(\OmegaR)\ \Ra\ \uvel,\grad\uvel\in\LRper{\frac{3}{2}s}(\OmegaR).
\end{align}
We now turn to estimate \eqref{lin_estofpressureLem_est1} of the pressure term. By H\"older's inequality,  
\begin{align*}
&\int_0^\per \Bp{\norm{\grad\uvel(t,\cdot)}_{s,\Omega_\Rzero}^{\frac{s-1}{s}}\,\norm{\grad\uvel(t,\cdot)}_{1,s,\Omega_\Rzero}^\frac{1}{s}}^{\frac{3}{2}s}\,\dt \leq \Bp{\int_0^\per  \norm{\grad\uvel(t,\cdot)}_{s,\Omega_\Rzero}^{\frac{ (s-1)s}{\frac{2}{3}s-1}}\,\dt}^{\frac{s-\frac{3}{2}}{s}}
\norm{\uvel}_{1,2,s}^{\frac{3}{2}}.
\end{align*}
Utilizing again Theorem \ref{SobEmbeddingThm}, this time with $\beta=1$, we see that the right-hand side above is finite for all $s\in[2,\infty)$, provided $\uvel\in\WSRper{1,2}{s}(\OmegaR)$. 
Due to the normalization of the pressure $\upres$ carried out in the beginning of the proof, Lemma \ref{lin_estofpressureLem} can be applied to infer from \eqref{lin_estofpressureLem_est1} that
\begin{align}\label{lin_maxregcomplspacesThm_bsregpres}
\forall s\in[2,\infty):\quad \uvel\in\WSRper{1,2}{s}(\OmegaR)\ \Ra\ \upres\in\LRper{\frac{3}{2}s}(\R\times\Omega_\Rzero).
\end{align}
Combining \eqref{lin_maxregcomplspacesThm_aprioriEstfinal} with the implications \eqref{lin_maxregcomplspacesThm_bsregvel} and
\eqref{lin_maxregcomplspacesThm_bsregpres}, we find that
\begin{align}\label{lin_maxregcomplspacesThm_bsfinal}
\forall s\in[2,\infty):\quad \uvel\in\WSRper{1,2}{s}(\OmegaR)\ \Ra\ \uvel\in\WSRper{1,2}{\frac{3}{2}s}(\OmegaR).
\end{align}
Starting with $s=2$, we can now boot-strap \eqref{lin_maxregcomplspacesThm_bsfinal} a sufficient number of times to deduce that 
$\uvel\in\WSRper{1,2}{s}(\OmegaR)$ for any $s\in(2,\infty)$. Knowing now that the right-hand side of
\eqref{lin_maxregcomplspacesThm_aprioriEstfinal} is finite for all $s\in(1,\infty)$, we can use interpolation and
\eqref{lin_estofpressureLem_est1}
in combination with Young's inequality to deduce
\begin{align*}
\norm{\uvel}_{1,2,s}
\leq\Cc{c}\,\bp{\norm{F}_s + 
\norm{\uvel}_{s,\torus\times\Omega_\Rzero}}
\end{align*} 
for all $s\in(1,\infty)$. It then follows directly from \eqref{lin_linsystem} that 
\begin{align}\label{lin_maxregcomplspacesThm_apriorionlylowestorder}
\norm{\uvel}_{1,2,s} + \norm{\grad\upres}_s
\leq\Cclast{c}\,\bp{\norm{F}_s + 
\norm{\uvel}_{s,\torus\times\Omega_\Rzero}}
\end{align} 
for all $s\in(1,\infty)$. 

If $(\Uvel,\Upres)\in\WSRcomplper{1,2}{r}(\OmegaR)\times\DSRcomplper{1}{r}(\R\times\Omega)$ with $r\in(1,\infty)$ is a solution to \eqref{lin_linsystem} with homogeneous right-hand side, then $\Uvel=\grad\Upres=0$ follows by a duality argument. 
More specifically, since  
for arbitrary $\testphi\in\CRcicomplper(\OmegaR)^3$ 
existence of a solution $(\Wvel,\Wpres)\in\WSRcomplper{1,2}{r'}(\OmegaR)\times\DSRcomplper{1}{r'}(\R\times\Omega)$
to \eqref{lin_linsystem} with $\testphi$ as the right-hand side has just been established, the computation
\begin{align}\label{lin_maxregcomplspacesThm_dualitycomputation}
\begin{aligned}
0 
&= \int_0^\per\int_\Omega (\partial_t\Uvel - \Delta\Uvel + \rey\partial_1\Uvel  + \grad\Upres)\cdot\Wvel\,\dx\dt\\
&= -\int_0^\per\int_\Omega \Uvel\cdot(\partial_t\Wvel  - \Delta\Wvel + \rey\partial_1\Wvel + \grad\Wpres)\,\dx\dt = \int_0^\per\int_\Omega \Uvel\cdot \testphi\,\dx\dt
\end{aligned}
\end{align}
is valid.
It follows that $\Uvel=0$ and in turn, directly from \eqref{lin_linsystem}, that also $\grad\Upres=0$.

We now return to the estimate \eqref{lin_maxregcomplspacesThm_apriorionlylowestorder}. Owing to the fact a solution 
 to \eqref{lin_linsystem} with homogeneous right-hand is necessarily zero, which 
we have just shown above, 
a standard contradiction argument, see for example \cite[Proof of Proposition 2]{GaldiTPFlowMovingCylinder}, can be used to eliminate the lower order term on 
the right-hand side in \eqref{lin_maxregcomplspacesThm_apriorionlylowestorder} to conclude 
\begin{align}\label{lin_maxregcomplspacesThm_aprioriestfinalfinal}
\norm{\uvel}_{1,2,q}+\norm{\grad\upres}_{q}\leq\Cc{c}\,{\norm{F}_q }.
\end{align} 
It is easy to verify that $\CRcicomplper(\torus\times\Omega)$ is dense in $\LRcomplper{q}(\torus\times\Omega)$. 
By a density argument, the existence of a solution 
$(\uvel,\upres)\in\WSRcomplper{1,2}{q}(\OmegaR)^3\times\DSRcomplper{1}{q}(\torus\times\Omega)$ to \eqref{lin_linsystem} that satisfies
\eqref{lin_maxregcomplspacesThm_aprioriestfinalfinal} follows for any $F\in\LRcomplper{q}(\torus\times\Omega)^3$. 

Finally, assume $(\tuvel,\tupres)\in\WSRcomplper{1,2}{r}(\OmegaR)\times\DSRcomplper{1}{r}(\OmegaR)$ is another 
solution to \eqref{lin_linsystem} with $r\in(1,\infty)$. 
The duality argument
used in \eqref{lin_maxregcomplspacesThm_dualitycomputation} applied to the difference $(\uvel-\tuvel,\upres-\tupres)$ yields $\uvel=\tuvel$ and $\grad\upres=\grad\tupres$. The proof of theorem 
is thereby complete. 
\end{proof}

By combining Theorem \ref{lin_maxregcomplspacesThm} with well-known $\LR{q}$ estimates for the steady-state Oseen system, we can formulate what can be referred to as 
``maximal $\LR{q}$ regularity'' of the time-periodic Oseen system \eqref{lin_linsystem}.

\begin{cor}\label{ext_lqestThmFull}
Let $\Omega\subset\R^3$ be an exterior domain of class $\CR{2}$, $\rey\in(0,\rey_0]$ and $q\in(1,2)$. 
Define 
\begin{align*}
&\xoseensigmazero{q}(\Omega):=\setc{\vvel\in\xoseen{q}(\Omega)}{\Div\vvel = 0,\ \vvel=0\text{ on }\partial\Omega}.
\end{align*}
Moreover, let $r\in(1,\infty)$ and 
\begin{align*}
\WSRpercomplsigmazero{1,2}{q,r}(\OmegaR):=\setcl{\wvel\in\WSRcomplper{1,2}{q,r}(\OmegaR)}{\Div\wvel=0,\ \wvel=0 \text{ on }\partial\Omega}.
\end{align*}
Then the $\per$-time-periodic Oseen operator
\begin{align}\label{ext_lqestThmFull_OseenOprDef}
\begin{aligned}
&\ALOseen:\Bp{\xoseensigmazero{q}(\Omega)\oplus\WSRpercomplsigmazero{1,2}{q,r}(\OmegaR)}\times 
\Bp{\DSRRzero{1}{q}(\Omega)\oplus\DSRcomplperRzero{1}{q,r}(\OmegaR)} \ra \\
&\qquad\qquad\qquad\qquad\qquad\qquad\qquad\qquad\qquad\qquad\qquad\qquad\qquad\LR{q}(\Omega)\oplus\LRcomplper{q,r}(\OmegaR),\\
&\ALOseen(\vvel+\wvel,\vpres+\wpres):= \partial_t\wvel  -\nu\Delta(\vvel+\wvel) + \lambda\partial_1(\vvel+\wvel) + \grad(\vpres+\wpres)
\end{aligned}
\end{align}
is a homeomorphism with $\norm{\ALOseeninverse}$ depending only on $q$, $r$, $\Omega$, $\nu$ and $\rey$.
If $q\in(1,\frac{3}{2})$, then  $\norm{\ALOseeninverse}$ depends only on the upper bound $\rey_0$ and \emph{not} on 
$\rey$ itself.
\end{cor}

\begin{proof}
It is well-known that the steady-state Oseen operator, that is, the Oseen operator from \eqref{ext_lqestThmFull_OseenOprDef} 
restricted to time-independent functions, is a homeomorphism as a mapping
$\ALOseen:\xoseensigmazero{q}(\Omega)\times \DSRRzero{1}{q}(\Omega)\ra\LR{q}(\Omega)$; 
see for example \cite[Theorem VII.7.1]{GaldiBookNew}.
By Theorem \ref{lin_maxregcomplspacesThm}, it follows that also the time-periodic Oseen operator is a homeomorphism as a mapping 
$\ALOseen:\WSRpercomplsigmazero{1,2}{q,r}(\OmegaR)\times\DSRcomplperRzero{1}{q,r}(\OmegaR)\ra\LRcomplper{q,r}(\OmegaR)$.  
Since clearly $\proj$ and $\projcompl$ commute with $\ALOseen$, it further follows that $\ALOseen$ is a homeomorphism as an operator in the setting 
\eqref{ext_lqestThmFull_OseenOprDef}. 
The dependency of $\norm{\ALOseeninverse}$ on the various parameters follows from \cite[Theorem VII.7.1]{GaldiBookNew} and Theorem \ref{lin_maxregcomplspacesThm}. 
\end{proof}

\section{Fully Nonlinear Problem}

Existence of a solution to the fully nonlinear problem \eqref{intro_nspastbody} shall now be established. 
We employ a fixed point argument based on the
estimates established for the linearized system \eqref{lin_linsystem} in the previous section. For this approach to work, we need to assume 
\eqref{intro_uvelinftyneqzero} to ensure that \eqref{lin_linsystem} is indeed a suitable linearization of \eqref{intro_nspastbody}.
Moreover, we need to assume \eqref{intro_bodyvelsingleaxis} to ensure that all terms appearing on the right-hand side after linearizing 
\eqref{intro_nspastbody} are subordinate to the linear operator on the left-hand side.

For convenience, we rewrite \eqref{intro_nspastbody} by replacing $\uvel$ with $\uvel+\bodyvel$ and obtain the following equivalent problem:
\begin{align}\label{existence_nspastbody}
\begin{pdeq}
&\partial_t\uvel + \nsnonlinb{(\uvel-\bodyvel)}{\uvel} = \nu\Delta\uvel - \grad\upres + f && \tin\OmegaR,\\
&\Div\uvel =0 && \tin\OmegaR,\\
&\uvel=\uvelbrd&& \ton\partialOmegaR,\\
&\lim_{\snorm{x}\ra\infty} \uvel(t,x) = 0.
\end{pdeq}
\end{align}
 
\begin{thm}\label{ext_ExistenceNonlinProbThm}
Let $\Omega\subset\R^3$ be an exterior domain of class $\CR{2}$.
Assume that $\bodyvel(t)=-\uvelinfty(t)\eone$ for a $\per$-periodic function $\uvelinfty:\R\ra\R$ with $\rey:=\proj\uvelinfty>0$.
Moreover, let $q\in\big[\frac{6}{5},\frac{4}{3}\big]$. 
There is an $\rey_0>0$ such that for all $\rey\in(0,\rey_0]$ there is an 
$\Cc[ext_ExistenceNonlinProbThmEpsilon]{eps}>0$ such that for all $\f\in\LRper{q}(\OmegaR)^3$, $\uvelinfty\in\LRper{\infty}(\R)$ and 
\begin{align}\label{ext_ExistenceNonlinProbThmTracespace}
\uvelbrd\in\WSRper{1}{q}\bp{\R;\WSR{2-\frac{3-q}{3q}}{\frac{3q}{3-q}}(\partial\Omega)}^3
\end{align}
satisfying
\begin{align}\label{ext_ExistenceNonlinProbThm_SmallnessCond3D}
\norm{\f}_q + \norm{\projcompl \f}_{\frac{3q}{3-q}} + \norm{\projcompl\uvelinfty}_{\infty} + 
\norm{\uvelbrd}_{\WSRper{1}{q}\bp{\R;\WSR{2-\frac{3-q}{3q}}{\frac{3q}{3-q}}(\partial\Omega)}} \leq \const{ext_ExistenceNonlinProbThmEpsilon}
\end{align}
there is a solution 
\begin{align}\label{ext_ExistenceNonlinProbThm_SolClass3D}
(\uvel,\upres)\in
\Bp{\xoseen{q}(\Omega)\oplus\WSRcomplper{1,2}{q,\frac{3q}{3-q}}(\OmegaR)}^3\times 
\Bp{\DSR{1}{q}(\Omega)\oplus\DSRcomplper{1}{q,\frac{3q}{3-q}}(\OmegaR)}
\end{align}
to \eqref{existence_nspastbody}.
\end{thm}

\begin{proof}
In order to ``lift'' the boundary values in \eqref{existence_nspastbody}, that is, rewrite the system as one of homogeneous boundary values, a solution  
$(\liftWvel,\liftWpres)\in\WSRcomplper{1,2}{q,\frac{3q}{3-q}}(\OmegaR)\times\DSRcomplper{1}{q,\frac{3q}{3-q}}(\OmegaR)$ to 
\begin{align}\label{ext_ExistenceNonlinProbThm_liftW}
\begin{pdeq}
-\nu\Delta\liftWvel + \grad\liftWpres &= \liftWvel && \tin\OmegaR,\\
\Div\liftWvel &= 0 && \tin\OmegaR,\\
\liftWvel &=\projcompl\uvelbrd && \ton \partialOmegaR,
\end{pdeq}
\end{align}
is introduced.
Observe that \eqref{ext_ExistenceNonlinProbThm_liftW} is a Stokes resolvent-type problem. One can therefore 
use standard methods to solve \eqref{ext_ExistenceNonlinProbThm_liftW} in $\per$-time periodic function spaces and obtain a solution that satisfies 
\begin{align}\label{ext_ExistenceNonlinProbThm_liftWest}
\forall r\in\Bpb{1,\frac{3q}{3-q}}:\ \norm{\liftWvel}_{1,2,r} + \norm{\grad\liftWpres}_{r} \leq \Cc{c} 
\norm{\uvelbrd}_{\WSRper{1}{q}\bp{\R;\WSR{2-\frac{3-q}{3q}}{\frac{3q}{3-q}}(\partial\Omega)}},
\end{align}   
where $\Cclast{c}=\Cclast{c}(r,q,\Omega,\nu)$. Furthermore, classical results for the steady-state Oseen problem \cite[Theorem VII.7.1]{GaldiBookNew}
ensure existence of a solution $(\liftVvel,\liftVpres)\in\xoseen{q}(\Omega)\times\DSR{1}{q}(\Omega)$ to    
\begin{align}\label{ext_ExistenceNonlinProbThm_liftV}
\begin{pdeq}
-\nu\Delta\liftVvel + \rey\partial_1\liftVvel + \grad\liftVpres &= 0 && \tin\Omega,\\
\Div\liftVvel &= 0 && \tin\Omega,\\
\liftVvel &=\proj\uvelbrd && \ton \partial\Omega,
\end{pdeq}
\end{align}
which satisfies, since $q\geq\frac{6}{5}$ implies $\frac{3q}{3-q}\geq 2$,
\begin{align}\label{ext_ExistenceNonlinProbThm_liftVest}
\forall r\in(1,2):\ \norm{\liftVvel}_{\xoseen{r}(\Omega)} + \norm{\grad\liftVpres}_{r}
\leq \Cc{c} \norm{\proj\uvelbrd}_{\WSR{2-\frac{3-q}{3q}}{\frac{3q}{3-q}}(\partial\Omega)},
\end{align}
where $\Cclast{c}=\Cclast{c}(r,\Omega,\nu)$.  
We shall now establish existence of a solution $(\uvel,\upres)$ to \eqref{existence_nspastbody} on the form 
\begin{align}\label{ext_ExistenceNonlinProbThm_solform}
\uvel=\vvel+\liftVvel+\wvel+\liftWvel,\quad 
\upres=\vpres+\liftVpres+\wpres+\liftWpres,
\end{align}
where $(\vvel,\vpres)\in\xoseen{q}(\Omega)\times\DSR{1}{q}(\Omega)$ is a solution
to the steady-state problem
\begin{align}\label{ext_ExistenceNonlinProbThm_vveleq}
\begin{pdeq}
-\nu\Delta\vvel + \rey\partial_1\vvel + \grad\vpres &= \rhsvvel(\vvel,\wvel,\liftVvel,\liftWvel) && \tin\Omega,\\
\Div\vvel &= 0 && \tin\Omega,\\
\vvel &=0 && \ton \partial\Omega,
\end{pdeq}
\end{align}
with
\begin{align*}
\rhsvvel(\vvel,\wvel,\liftVvel,\liftWvel):=
&-\nsnonlin{\vvel}
-\nsnonlinb{\vvel}{\liftVvel} 
-\nsnonlinb{\liftVvel}{\vvel}
-\nsnonlinb{\liftVvel}{\liftVvel} \\
&-\proj\bb{\nsnonlin{\wvel}}
-\proj\bb{\nsnonlinb{\wvel}{\liftWvel}}
-\proj\bb{\nsnonlinb{\liftWvel}{\wvel}}
-\proj\bb{\nsnonlinb{\liftWvel}{\liftWvel}}\\
&-\proj\bb{\projcompl\uvelinftyscalar\, \partial_1\wvel}
-\proj\bb{\projcompl\uvelinftyscalar\, \partial_1\liftWvel}
+ \proj f,
\end{align*}
and $(\wvel,\wpres)\in\WSRcomplper{1,2}{q,\frac{3q}{3-q}}(\OmegaR)\times \DSRcomplper{1}{q,\frac{3q}{3-q}}(\OmegaR)$ a solution to 
\begin{align}\label{ext_ExistenceNonlinProbThm_wveleq}
\begin{pdeq}
\partial_t\wvel-\nu\Delta\wvel + \rey\partial_1\wvel + \grad\wpres &= \rhswvel(\vvel,\wvel,\liftVvel,\liftWvel) && \tin\OmegaR,\\
\Div\wvel &= 0 && \tin\OmegaR,\\
\wvel &=0 && \ton \partialOmegaR,
\end{pdeq}
\end{align}
with 
\begin{align*}
\rhswvel(\vvel,\wvel,\liftVvel,\liftWvel):=
&
-\projcompl\bb{\nsnonlinb{\wvel}{\wvel}}
-\projcompl\bb{\nsnonlinb{\wvel}{\liftWvel}}
-\projcompl\bb{\nsnonlinb{\liftWvel}{\wvel}}
-\projcompl\bb{\nsnonlinb{\liftWvel}{\liftWvel}}
\\
&
-\nsnonlinb{\vvel}{\wvel}
-\nsnonlinb{\vvel}{\liftWvel}
-\nsnonlinb{\wvel}{\vvel}
-\nsnonlinb{\wvel}{\liftVvel}
\\
&
-\nsnonlinb{\liftVvel}{\wvel}
-\nsnonlinb{\liftVvel}{\liftWvel}
-\nsnonlinb{\liftWvel}{\vvel}
-\nsnonlinb{\liftWvel}{\liftVvel}
\\
&
-\projcompl\uvelinftyscalar\,\partial_1\vvel
-\projcompl\uvelinftyscalar\,\partial_1\liftVvel
-\projcompl\bb{\projcompl\uvelinftyscalar\,\partial_1\wvel}
-\projcompl\bb{\projcompl\uvelinftyscalar\,\partial_1\liftWvel} 
\\
&-\partial_t\liftWvel-\liftWvel+\rey\partial_1\liftWvel+\projcompl f. 
\end{align*}
The systems \eqref{ext_ExistenceNonlinProbThm_vveleq} and \eqref{ext_ExistenceNonlinProbThm_wveleq} emerge as the result of inserting 
\eqref{ext_ExistenceNonlinProbThm_solform} into \eqref{existence_nspastbody} and subsequently applying first $\proj$ then $\projcompl$ to the equations.
Recalling the function spaces introduced in Corollary \ref{ext_lqestThmFull}, we define the Banach space
\begin{align}\label{ext_ExistenceNonlinProbThm_DefKspace}
\fixedpointspace{q}(\OmegaR):=\xoseensigmazero{q}(\Omega)\oplus\WSRpercomplsigmazero{1,2}{q,\frac{3q}{3-q}}(\OmegaR)\times 
\DSRRzero{1}{q}(\Omega)\oplus\DSRcomplperRzero{1}{q,\frac{3q}{3-q}}(\OmegaR).
\end{align}
We can obtain solutions $(\vvel,\vpres)$ and $(\wvel,\wpres)$ to 
\eqref{ext_ExistenceNonlinProbThm_vveleq} and
\eqref{ext_ExistenceNonlinProbThm_wveleq}, respectively,
as a fixed point of the mapping
\begin{align*}
&\fixedpointmap:\fixedpointspace{q}(\OmegaR)\ra\fixedpointspace{q}(\OmegaR),\\
&\fixedpointmap(\vvel+\wvel,\vpres+\wpres):=
\ALOseeninverse\bp{\rhsvvel(\vvel,\wvel,\liftVvel,\liftWvel)+\rhswvel\np{\vvel,\wvel,\liftVvel,\liftWvel}}.
\end{align*}
We shall show that $\fixedpointmap$ is a contracting self-mapping on a ball of sufficiently small radius. For 
this purpose, let $\rho>0$ and consider some $(\vvel+\wvel,\vpres+\wpres)\in\fixedpointspace{q}\cap\B_\rho$.
Suitable estimates of $\rhsvvel$ and $\rhswvel$ in combination with a smallness assumption on $\const{ext_ExistenceNonlinProbThmEpsilon}$
from \eqref{ext_ExistenceNonlinProbThm_SmallnessCond3D} are needed to guarantee that $\fixedpointmap$ has the desired properties. 
We first estimate the terms of $\rhsvvel$. Since $q\in[\frac{6}{5},\frac{4}{3}]$ implies $\frac{4q}{4-q}\leq 2 \leq \frac{3q}{3-q}$, 
we can employ first H\"older's inequality and then interpolation to estimate
\begin{align*}
\norm{\nsnonlin{\vvel}}_q &\leq \norm{\vvel}_{\frac{2q}{2-q}}\norm{\grad\vvel}_2
\leq \rey^{-\half}\, \norm{\vvel}_{\xoseen{q}}\,\norm{\grad\vvel}_{\frac{4q}{4-q}}^\theta\,\norm{\grad\vvel}_{\frac{3q}{3-q}}^{1-\theta} 
\end{align*}
with $\theta=\frac{10q-12}{q}$. It thus follows by the Sobolev embedding $\WSR{1}{q}(\Omega)\embeds\LR{\frac{3q}{3-q}}(\Omega)$ that 
\begin{align}\label{ext_ExistenceNonlinProbThm_nsnonlintermest}
\begin{aligned}
\norm{\nsnonlin{\vvel}}_q 
\leq \Cc{c}\,\rey^{-\half-\frac{\theta}{4}}\, \norm{\vvel}_{\xoseen{q}}^{1+\theta}\,\norm{\grad^2\vvel}_{q}^{1-\theta} 
&\leq \Cclast{c}\,\rey^{-\frac{3q-3}{q}}\norm{\vvel}_{\xoseen{q}}^2 \leq \Cclast{c}\rey^{-\frac{3q-3}{q}}\rho^2.
\end{aligned}
\end{align}
The other terms in the definition of $\rhsvvel$ can be estimated in a similar fashion to conclude in combination with  
\eqref{ext_ExistenceNonlinProbThm_SmallnessCond3D} that
\begin{align*}
\norm{\rhsvvel(\vvel,\wvel,\liftVvel,\liftWvel)}_{\LR{q}(\Omega)}\leq\Cc{c}
\bp{\rey^{-\frac{3q-3}{q}}\rho^2 +  \rey^{-\frac{1}{2}}\rho\const{ext_ExistenceNonlinProbThmEpsilon} +\rho\const{ext_ExistenceNonlinProbThmEpsilon}+\const{ext_ExistenceNonlinProbThmEpsilon}^2+\const{ext_ExistenceNonlinProbThmEpsilon}}.  
\end{align*}
An estimate of $\rhswvel$ is required both in the $\LRper{q}(\OmegaR)$ and $\LRper{\frac{3q}{3-q}}(\OmegaR)$ norm. 
Observe that  
\begin{align}\label{ext_ExistenceNonlinProbThm_noteworthy}
\norm{\projcompl\uvelinftyscalar\,\partial_1\vvel}_{\LRper{q}(\OmegaR)} \leq \Cc{c} \norm{\projcompl\uvelinftyscalar}_\infty \norm{\partial_1\vvel}_{q}
\leq \Cclast{c}\rey^{-1}\const{ext_ExistenceNonlinProbThmEpsilon}\rho. 
\end{align}
With the help of the embedding properties in Theorem \ref{SobEmbeddingThm},
the other terms in $\rhswvel$ can be estimated to obtain
\begin{align*}
\norm{\rhswvel(\vvel,\wvel,\liftVvel,\liftWvel)}_{\LRper{q}(\OmegaR)}\leq\Cc{c}
\bp{\rey^{-1}\const{ext_ExistenceNonlinProbThmEpsilon}\rho+ \rho^2+\rho\const{ext_ExistenceNonlinProbThmEpsilon}+\const{ext_ExistenceNonlinProbThmEpsilon}^2+\rey\const{ext_ExistenceNonlinProbThmEpsilon}+\const{ext_ExistenceNonlinProbThmEpsilon}}.  
\end{align*}
The embedding properties can also be used to establish an $\LRcomplper{\frac{3q}{3-q}}(\OmegaR)$ estimate of $\rhswvel$. For example, 
\begin{align*}
\norm{\nsnonlinb{\wvel}{\wvel}}_{\LRper{\frac{3q}{3-q}}(\OmegaR)}\leq 
\Cc{c}\norm{\wvel}_{\LRper{\frac{3q}{3-q}}\np{\R;\LR{\infty}(\Omega)}}
\norm{\grad\wvel}_{\LRper{\infty}\np{\R;\LR{\frac{3q}{3-q}}(\Omega)}}
\leq \Cc{c}\rho^2,
\end{align*}
where Theorem \ref{SobEmbeddingThm} is utilized with $\alpha=0$ and $\beta=1$ in the last inequality.
For this particular utilization of Theorem \ref{SobEmbeddingThm}, it is required that $q\geq \frac{6}{5}$. 
Further note that 
\begin{align*}
\norm{\projcompl\uvelinftyscalar\,\partial_1\vvel}_{\frac{3q}{3-q}}
\leq \const{ext_ExistenceNonlinProbThmEpsilon}\norm{\vvel}_{\xoseen{q}}
\leq \const{ext_ExistenceNonlinProbThmEpsilon}\rho,
\end{align*}
which explains the choice of the exponent $\frac{3q}{3-q}$ in the setting of the mapping $\fixedpointmap$.
The rest of the terms in $\rhswvel$ can be estimated to conclude
\begin{align*}
\norm{\rhswvel(\vvel,\wvel,\liftVvel,\liftWvel)}_{\LRper{\frac{3q}{3-q}}(\OmegaR)}\leq\Cc{c}\bp{\rho^2+\rho\const{ext_ExistenceNonlinProbThmEpsilon}
+\const{ext_ExistenceNonlinProbThmEpsilon}^2+\rey\const{ext_ExistenceNonlinProbThmEpsilon}+\const{ext_ExistenceNonlinProbThmEpsilon}}.  
\end{align*}
We can now conclude from Corollary \ref{ext_lqestThmFull}, recall that  $\norm{\ALOseeninverse}$ is independent on $\rey$, the estimate
\begin{align*}
\norm{\fixedpointmap(\vvel+\wvel,\vpres+\wpres)}_{\fixedpointspace{q}}
&\leq \norm{\ALOseeninverse}\cdot \Bp{\norm{\rhsvvel}_{\LR{q}(\Omega)}+\norm{\rhswvel}_{\LRper{q,\frac{3q}{3-q}}(\OmegaR)}}\\
&\leq 
\Cc[ext_ExistenceNonlinProbThmFixedpointconst]{c}\bp{
\rey^{-\frac{3q-3}{q}}\rho^2 
+\rey^{-1}\const{ext_ExistenceNonlinProbThmEpsilon}\rho
+\rey^{-\frac{1}{2}}\rho\const{ext_ExistenceNonlinProbThmEpsilon}
+\rho\const{ext_ExistenceNonlinProbThmEpsilon}
+\const{ext_ExistenceNonlinProbThmEpsilon}^2
+\rey\const{ext_ExistenceNonlinProbThmEpsilon}
+\const{ext_ExistenceNonlinProbThmEpsilon}
}.
\end{align*}
In particular, $\fixedpointmap$ becomes a self-mapping on $\B_\rho$ if 
\begin{align*}
\const{ext_ExistenceNonlinProbThmFixedpointconst}\bp{
\rey^{-\frac{3q-3}{q}}\rho^2 
+\rey^{-1}\const{ext_ExistenceNonlinProbThmEpsilon}\rho
+\rey^{-\frac{1}{2}}\rho\const{ext_ExistenceNonlinProbThmEpsilon}
+\rho\const{ext_ExistenceNonlinProbThmEpsilon}
+\const{ext_ExistenceNonlinProbThmEpsilon}^2
+\rey\const{ext_ExistenceNonlinProbThmEpsilon}
+\const{ext_ExistenceNonlinProbThmEpsilon}
}
\leq \rho.
\end{align*}
One may choose $\const{ext_ExistenceNonlinProbThmEpsilon}:=\rey^2$ and $\rho:=\rey$ 
to find the above inequality satisfied for sufficiently small $\rey$. For such choice of parameters, one may further verify that $\fixedpointmap$ is also a contraction. 
By the contraction mapping principle, existence of a fixed point for $\fixedpointmap$ follows. This concludes the proof.
\end{proof}

\begin{rem}
As for classical anisotropic Sobolev spaces, the trace operator for time-periodic Sobolev spaces is continuous and surjective
for any $r\in(1,\infty)$ in the setting
\begin{align*}
\trace:\WSRper{1,2}{r}(\OmegaR)\ra\WSRper{1-\frac{1}{2r},2-\frac{1}{r}}{r}\np{\partialOmegaR}:=
\WSRper{1-\frac{1}{2r}}{r}\bp{\R;\LR{r}(\Omega)}\cap\LR{r}\bp{\R;\WSR{2-\frac{1}{r}}{r}(\Omega)}.
\end{align*} 
Consequently, for the Sobolev space in \eqref{ext_ExistenceNonlinProbThm_SolClass3D}, in which a solution $\uvel$ is established in 
Theorem \ref{ext_ExistenceNonlinProbThm}, we find that
\begin{align}\label{nonlin_maxregtracespace}
\trace:  \WSRcomplper{1,2}{q,\frac{3q}{3-q}}(\OmegaR) \ra \WSRper{1-\frac{3-q}{6q},2-\frac{3-q}{3q}}{\frac{3q}{3-q}}\np{\partialOmegaR}
\end{align}
is continuous and surjective. It therefore seems natural that Theorem \ref{ext_ExistenceNonlinProbThm} would hold for arbitrary boundary values $\uvelbrd$ 
in the space on the right-hand side in \eqref{nonlin_maxregtracespace}. In other words, it seems we are requiring 
too much regularity on $\uvelbrd$ in \eqref{ext_ExistenceNonlinProbThmTracespace}. We leave as an open question as to whether or not the regularity assumptions on the boundary values can be weakened. 
\end{rem}

\bibliographystyle{abbrv}

\end{document}